\documentclass[a4paper,11pt,reqno]{amsart}
\usepackage[T1]{fontenc}
\usepackage[utf8]{inputenc}
\usepackage{lmodern}
\usepackage{amsmath, amsthm, amssymb,amscd, mathrsfs, amsfonts}
\usepackage{hyperref}
\usepackage{times}
\usepackage[all]{xy}
\usepackage{todonotes}
\usepackage{xcolor}
\usepackage{amsthm}                
\usepackage[margin=1in]{geometry}  
\usepackage{graphics}
\usepackage{epstopdf}
\usepackage{graphicx}
\usepackage{subfigure}
\usepackage{latexsym}
\usepackage{tikz}
\usepackage{mathdots}

\usepackage{hyperref}
\usepackage{comment}

\usepackage[lite]{amsrefs}

\renewcommand{\PrintDOI}[1]{\href{http://dx.doi.org/\detokenize{#1}}{doi: \detokenize{#1}}%
	\IfEmptyBibField{pages}{, (to appear in print)}{}}

\newtheorem{theorem}{Theorem}[section]

\newtheorem{corollary}[theorem]{Corollary}
\newtheorem{proposition}[theorem]{Proposition}

\theoremstyle{definition}
\newtheorem{definition}[theorem]{Definition}
\newtheorem{example}[theorem]{Example}
\theoremstyle{remark}
\newtheorem{remark}[theorem]{Remark}
\numberwithin{equation}{section}

\begin{document}

\title{Framed Knots}

\author{Mohamed Elhamdadi} 
\address{Department of Mathematics, 
University of South Florida, Tampa, FL 33620, U.S.A.} 
\email{emohamed@math.usf.edu} 

\author{Mustafa Hajij}
\address{Department of Computer Science, Ohio State University, 
Columbus, OH 43210 USA}
\email{mhajij.1@osu.edu}

\author{Kyle Istvan}
\address{}
\email{KyleIstvan@gmail.com}

\begin{abstract} This article gives a foundational account of various characterizations of framed links in the $3$-sphere.  
\end{abstract}

\maketitle

\tableofcontents

\section{Introduction}
The field of knot theory has been a fundamental area of research for over 150 years. Framed knots are an extension that we can visualize as closed loops of knotted flat ribbons.  Interest  in framed knots follows from the critical role they play in  low dimensional topology. For example, a foundational result in 
the theory of $3$-manifolds, the Lickorish-Wallace Theorem \cite{Lickorish}, states that any closed, orientable, connected 3-manifold can be realized by performing a topological operation known as integer surgery on some framed link in the $3$-dimensional sphere $S^{3}.$  A diagrammatic process known as Kirby Calculus \cite{Kirby} allows us to determine homeomorphic equivalence of 3-manifolds given by such a description with only a couple simple moves on framed link diagrams, similar to the Reidemeister theorem for knots and links.  Framings also appear quite often when dealing with polynomial link invariants.   Framed links can even be used to encode handlebody decompositions of 4-manifolds, though we will restrict ourselves to 3-dimensional topological considerations in the pages to follow. The purpose of this article is to introduce the reader to various characterizations of framed knots and links, to show (or at least intuitively justify) their equivalence, and to discuss their application to $3$-manifold topology.  To that end, this text assumes the reader has basic understanding of algebraic topology. We will include proofs where appropriate, and provide citations whenever the complexity of such details falls outside the scope of this paper.  We will also provide auxiliary materials to help the reader visualize some of the more intuitive notions.

\section{Basic definitions, theorems, and conventions}


A \textit{knot} in the $3$-sphere $S^3$ is a smooth  one-to-one mapping $f:S^1\longrightarrow S^3$ (see Figure \ref{knot}). Equivalently, a knot can be be thought of as the set $f(S^1)$. We will work with these two definitions interchangeably.
A \textit{link} in $S^3$ is a finite collection of knots, called the \textit{components} of the
link, that do not intersect each other. Two links are considered to be equivalent if one can be deformed into the other without any one of the knots intersecting itself or any other knots\footnote{This is called ambient isotopy in the literature.}.
A \textit{link invariant} is a quantity, defined for each link in $S^3$, that 
 takes the same value for equivalent links\footnote{ The equivalence relation here is ambient isotopy.}. Link invariants play a fundamental role in low-dimensional topology.
In practice we usually work with a \textit{link diagram} of a link $L$. A link diagram is a projection of $L$ onto $S^2$ (resp. $\mathbb{R}^2$)
such that this projection has a finite number of non-tangentional intersection points, called crossings. Each crossing corresponds to exactly two points of the link $L$. See Figure \ref{knot} for an example. To store the relative spatial information in the crossings, we usually draw a small break in the projection of the strand closest to the projection sphere (resp. plane) to indicate that it crosses under the other strand.

\begin{figure}[h]
  \centering
   {\includegraphics[scale=0.085]{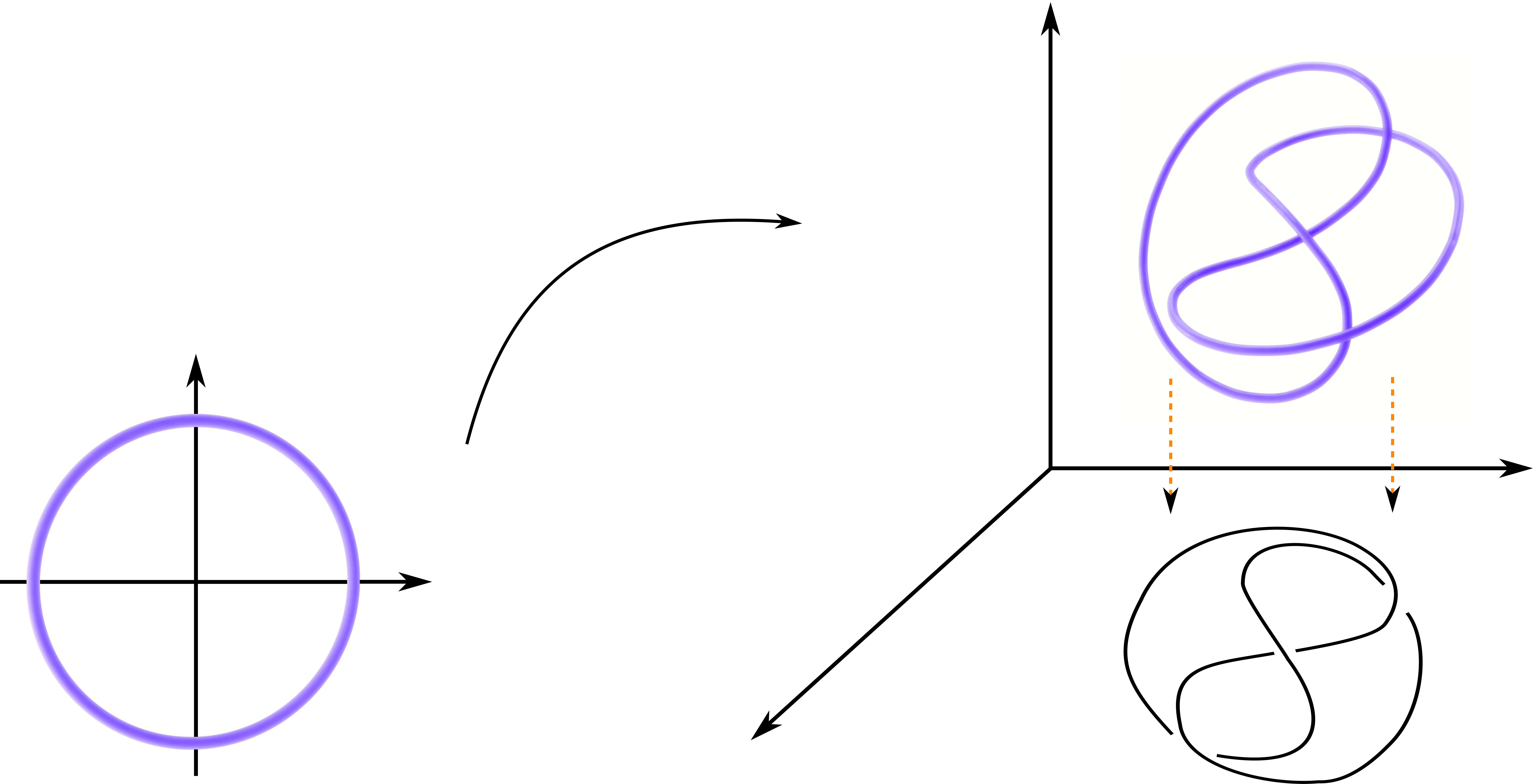}
\put(-330,95){$S^1$}
 \put(-240,150){$f$}
 \put(-170,150){$S^3$}
  \caption{A knot is a smooth-linear one-to-one mapping $f:S^1\longrightarrow S^3$. A knot diagram is a obtained by projecting this onto a plane. 
  }
  \label{knot}}
\end{figure}

A {\em{framed knot}} $(K,V)$ in $S^{3}$ is a knot $K$ equipped with
a continuous nonvanishing vector field $V$ normal to the knot, called a {\em{framing}} (see Figure~\ref{framed knot}). The magnitude of these vectors is largely irrelevant. 
Similarly, a {\em{framed link}} in $S^{3}$ is a link where each component knot is equipped with a framing.
A framed knot can be visualized
as a tangled ribbon that has had its two ends glued after an even number of half-twists, so as to yield an orientable surface. Note that this means we exclude the cases in which the ribbon is glued together after an odd number of half-twists, i.e. a Mobius band (see Figure~\ref{tangeld ribbon}). To put it it more precise terms, the ribbon forms an embedded annulus, one of whose boundary components is identified with the specified knot $K$.  For a given knot $K$, two framings on $K$ are considered to be equivalent if one can be transformed into the other by a smooth deformation \footnote{The precise notion of equivalence here is again ambient isotopy.}. This is indeed an equivalence relation on the set of framings, and as such the term ``framing'' will be used to refer to either an equivalence class or a representative vector field, as context dictates.  

Let us quickly demonstrate the equivalence of the definition of a framed knot and the conceptualization of the closed ribbon.  
 Given a framed knot $(K,V)$, we can construct a ribbon by pushing the knot $K$ along the vector field $V$, sweeping out an area.  Conversely, given a closed orientable ribbon in $S^{3}$, we can construct a framed knot by considering one of its boundary components to be the knot $K$, and choosing the vector field $V$ to lie in the ribbon, perpendicular to $K$ at every point of the knot.  The magnitude of the vector field is unimportant, and shall be ignored for the remainder of this text. (see Figure~\ref{all}).  
\begin{figure}[!h]
\centering
\begin{minipage}{.5\textwidth}
  \centering
  \includegraphics[width=.9\linewidth]{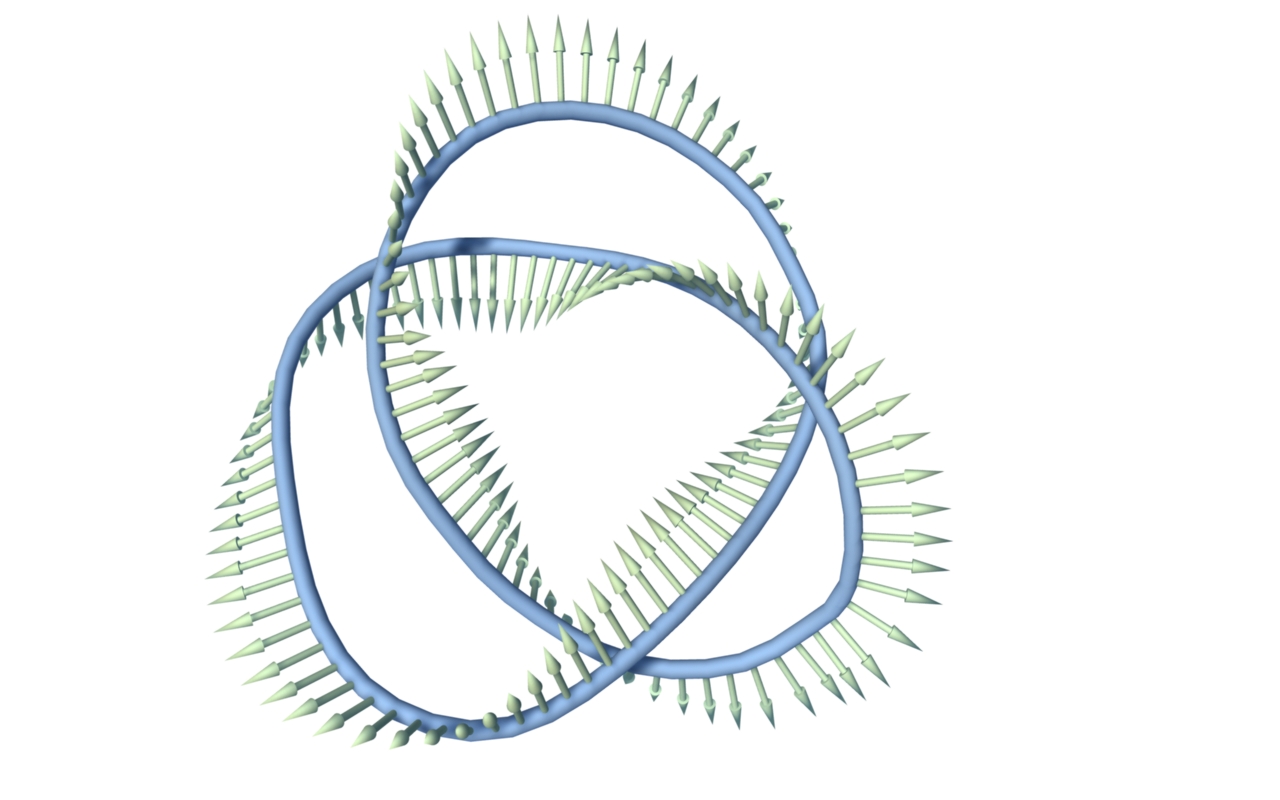}
  \caption{A framed Trefoil.}
  \label{framed knot}
\end{minipage}%
\begin{minipage}{.5\textwidth}
  \centering
  \includegraphics[width=.9\linewidth]{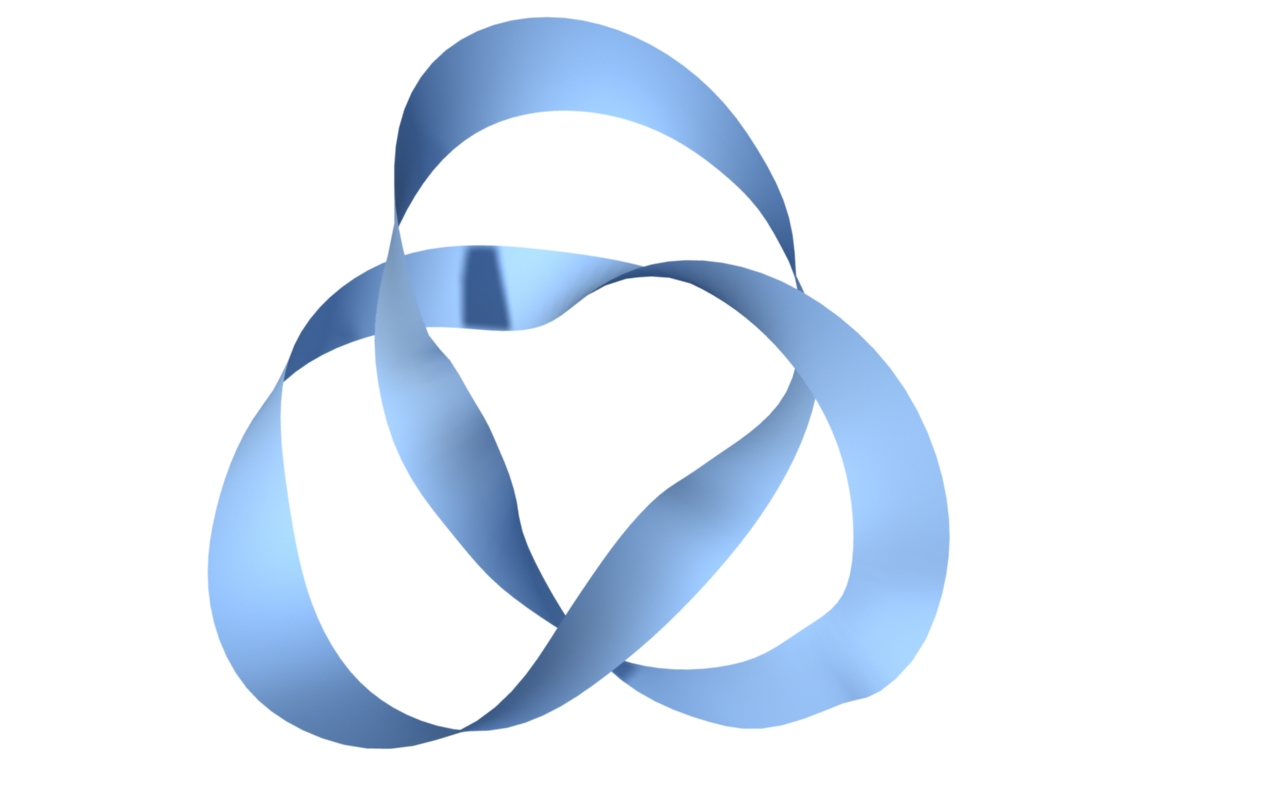}
  \caption{A tangled ribbon.}
  \label{tangeld ribbon}
\end{minipage}
\begin{minipage}{.5\textwidth}
  \centering
  \includegraphics[width=.9\linewidth]{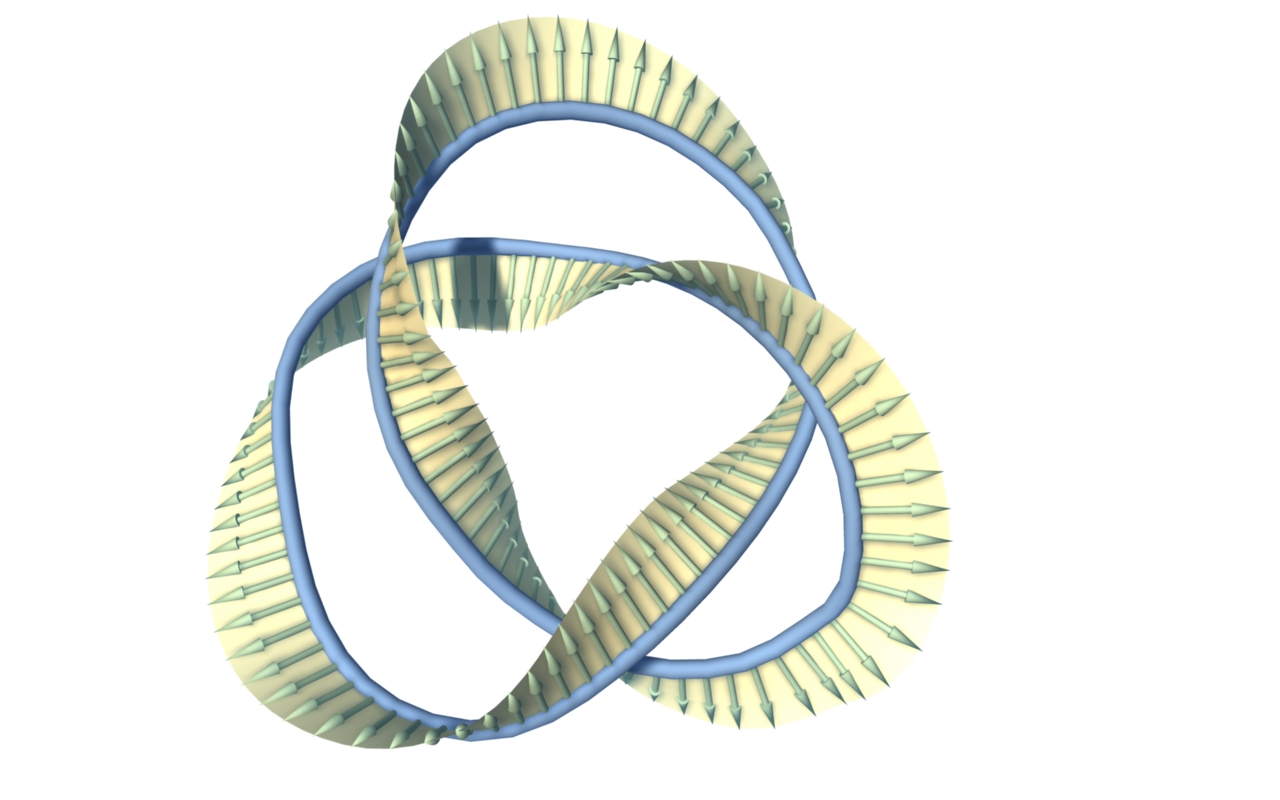}
  \caption{Framed knots $\leftrightarrow$ Tangled ribbons.}
  \label{all}
\end{minipage}
\end{figure}

Given a knot, one can define infinitely many framings on it. See also the \href{http://www.youtube.com/watch?v=KxEBhD0C2Pw}{framed knots movie representation}\footnote{The framed knot movie representation can be viewed by visiting http://www.youtube.com/watch?v=KxEBhD0C2Pw.}. 
Suppose that we are given a knot with a fixed framing.  One may obtain a new
framing from the existing one by cutting the ribbon and twisting it a nonzero integer multiple of $2\pi$ times around the knot, and then reconnecting the edges.
This operation leaves the knot itself fixed, and the reader should intuit that this is not a smooth deformation of the vector field.  It is in fact impossible to have {\em{any}} smooth deformation between these two vector fields, but this is more easily shown using some of the characterizations that follow.  

In the context of the previous operation, we see that the framing is associated with the number of ``twists'' the vector field performs around the knot, although it should not be immediately obvious how we can make such a definition precise.  How does one count the number of twists a vector field makes around an object that is itself tangled up in the $3$-sphere?  What accounts for a clockwise rotation, vs counterclockwise? As we will see, it is in fact possible to make such a definition, and knowing how many times the vector field is twisted around the knot allows one to completely determine the vector field up to a smooth deformation.  The equivalence class of the framing is
determined completely by this integer number of twists, called the \textit{{framing integer}}. Our next goal is to show how the framing integer can be easily computed from a diagram 
using the {\em{linking number}}.

\section{Writhe, linking and self-linking numbers}
\label{linking number}

In practice, knots and links are frequently represented via diagrams.  It is useful then to have a combinatorial (diagrammatic) method for computing the framing integer.  It turns out to be surprisingly easy to do, using the notion of the {\em{linking number}}. In what follows an \textit{oriented knot} is a knot which has been given an orientation. Similarly, an oriented link is a link for which each of its component has been given an orientation.
\begin{definition} \cite{Lickorish} \label{lk}
 Let $J$ and $K$ be two disjoint oriented knots, represented by a link diagram $D$.  The {\textit{linking number}} of $J$ and $K$, denoted $lk(J,K)$, is an integer, defined to be one half of the sum of the signs (see Figure \ref{sign}) of every crossing between $J$ and $K$ in the diagram $D$.  \end{definition}
 In this definition, we reiterate that self-crossings of the knots are not included in the summation.

\begin{figure}[htb]
  \centering
   {\includegraphics[scale=0.65]{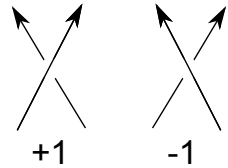}
  \caption{Positive and negative crossings. The sign of a crossing can be determined by the right hand rule.}
  \label{sign}}
\end{figure}

 It is easy to show that  the set of linking numbers is actually an invariant of links. The proof is a direct application of  Reidemeister’s theorem \cite{Reidemeister}, which says two diagrams of links represent the same link if and only if they are related through a finite sequence of three local moves called the Reidemeister moves of Figure \ref{Rmoves} and planar isotopy. We denote the three moves by $\Omega_{1}$, $\Omega_{2}$ and $\Omega_{3}$. Thus, to prove that the linking number is an invariant one needs only to check that the linking number does not change under the three Reidemeister moves 
  $\Omega_{1}$, $\Omega_{2}$ and $\Omega_{3}$. 

\begin{figure}[htb]
  \centering
   {\includegraphics[scale=0.13]{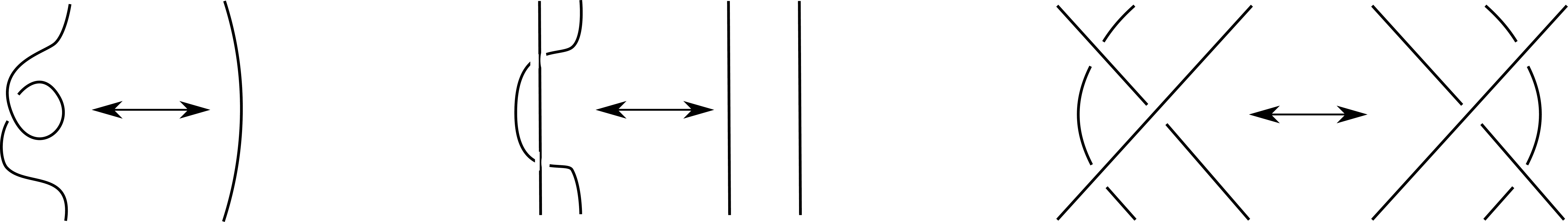}
     \put(-390,-23){$\Omega_1$}
 \put(-250,-23){$\Omega_2$}
 \put(-70,-23){$\Omega_3$}
  \caption{The Reidemeister moves $\Omega_1$, $\Omega_2$ and $\Omega_3$. Reidemeister’s theorem \cite{Reidemeister}, asserts that two diagrams of links represent the same link if and only if they are related through a finite sequence of the three  Reidemeister moves and planar isotopy.}
  \label{Rmoves}}
\end{figure}


For the invariance of the linking number under the first move, one can see immediately that $\Omega_{1}$ only adds or subtracts a self-crossing, and thus it leaves the summation unchanged.  On the other hand, performing the move $\Omega_{2}$ will either introduce or remove two crossings with opposite signs. This new pair of crossings is either two self-intersections of the same knot (which don't appear in the summation), or both occur between the two distinct knots, thus cancelling each other in the summation. Finally, for the invariance of the linking number under the third move $\Omega_{3}$, we note that the set of values being summed remains unchanged. See \href{https://www.youtube.com/watch?v=_qOl_5KcANE}{this linking number movie} \footnote{The linking number movie can be viewed by visiting $https://www.youtube.com/watch?v=\_qOl\_5KcANE$.}.  
Note also that for two knots $K_1$ and $K_2$, the definition implies
that
 $lk(K_1,K_2)=lk(K_2,K_1)$.

In order to state the next theorem we need to define the notion of  connected sum $K_1\# K_2$ of two knots $K_1$ and $K_2$.  This sum  is obtained by removing a single arc from each of the two knots, indicated by dotted lines in Figure \ref{connected sum}. The two augmented knots are then joined by adding arcs in $S^3 \backslash (K_1 \cup K_2)$ as indicated in the figure.  The union of the two new arcs and the two deleted arcs must bound a topological disc that intersects the original knots only along the deleted arcs.

\begin{figure}[h]
  \centering
   {\includegraphics[scale=0.13]{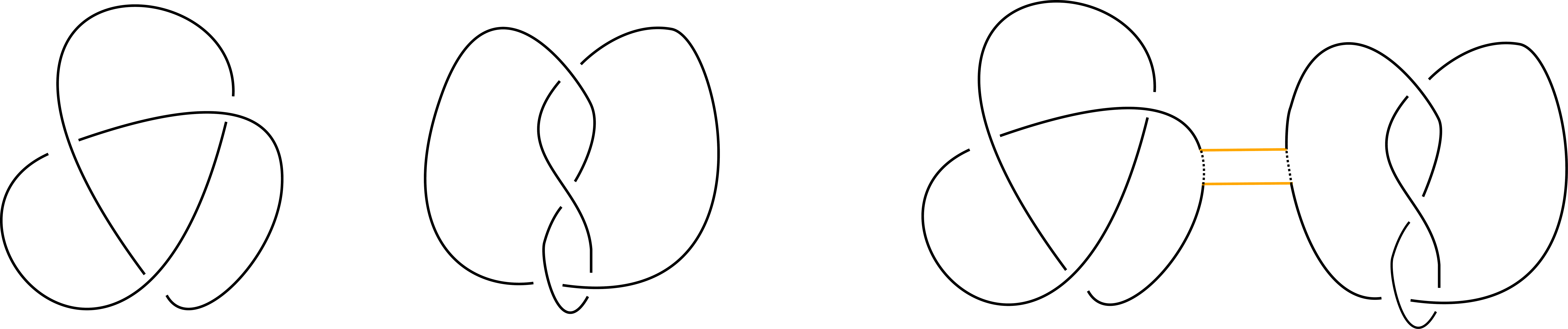}
   \put(-390,-15){$K_1$}
 \put(-290,-15){$K_2$}
 \put(-100,-15){$K_1 \# K_2$}
  \caption{Two knots $K_1$ and $K_2$ and their connected sum of $K_1\#K_2$ . The connected sum of two knots $K_1$ and $K_2$ is obtained by removing two arcs from the two knots $K_1$ and $K_2$. These deleted arcs are indicated in dotted line $K_1 \#K_2$. We then join
the two knots adding arcs as indicated in the figure of $K_1\#K_2$ }
  \label{connected sum}}
\end{figure}

In the following, the knot $K$ with the reversed orientation will be denoted by $-K$.  We have the following
\begin{theorem}
\label{linkingnumber}
Let $K_1$, $K_2$ and $K_3$ be three disjoint oriented knots in $S^3$. Then
\begin{enumerate}
\item
 $lk(K_{1}\# K_{2},K_{3})=lk(K_{1},K_{3})+lk(K_{2},K_{3})$.
\item
 $lk(K_{1},-K_{2})=lk(-K_{1},K_{2})=-lk(K_{1},K_{2})$.
\item
 Suppose that $K_2$ can be obtained from $K_1$ via a single crossing change. Then $lk(K_1,K_3)=lk(K_2,K_3)$. 
\end{enumerate}

\end{theorem}

\begin{proof}
  
  For item (1): any new crossings that are introduced in the connected sum occur in canceling pairs. Thus additivity follows directly from the definition of the linking number.

  Item (2) follows from the observation that by changing the orientation of only one of the pair, the sign of each crossing between the two knots changes.\\  Item (3) is immediate once we recall that self-crossings are not included in the calculation of linking number. 
\end{proof}

\begin{remark}

Item (3) of the previous theorem indicates that the linking number between two knots $J$ and $K$ is independent of
the knot types of $J$ and $K$. Applying part (3) to certain self-crossings in the knots diagram $J$ and $K$, one obtains eventually two trivial knots $J^{\prime}$ and $K^{\prime}$ that are linked together in such that $lk(K,J)=lk(K^{\prime},J^{\prime})$.  Figure ~\ref{linking2} shows the two illustrative possibilities of $K^{\prime}\cup J^{\prime}$.  

A careful inspection of the difference between the two possibilities shown in Figure \ref{linking2} hints at how we might define a notion of clockwise versus counterclockwise ``twisting'' of one knot about another.  Note that Item (2) of the previous theorem shows that any such definition can only be made relative to a choice of orientation on the knots.  In fact, we'll see that the most natural definition of the framing integer arises from a choice of orientation on a Seifert surface of the knot (although this choice is equivalent to choosing one for the knot itself).


\end{remark}  
\begin{figure}[htb]
  \centering
   {\includegraphics[scale=0.28]{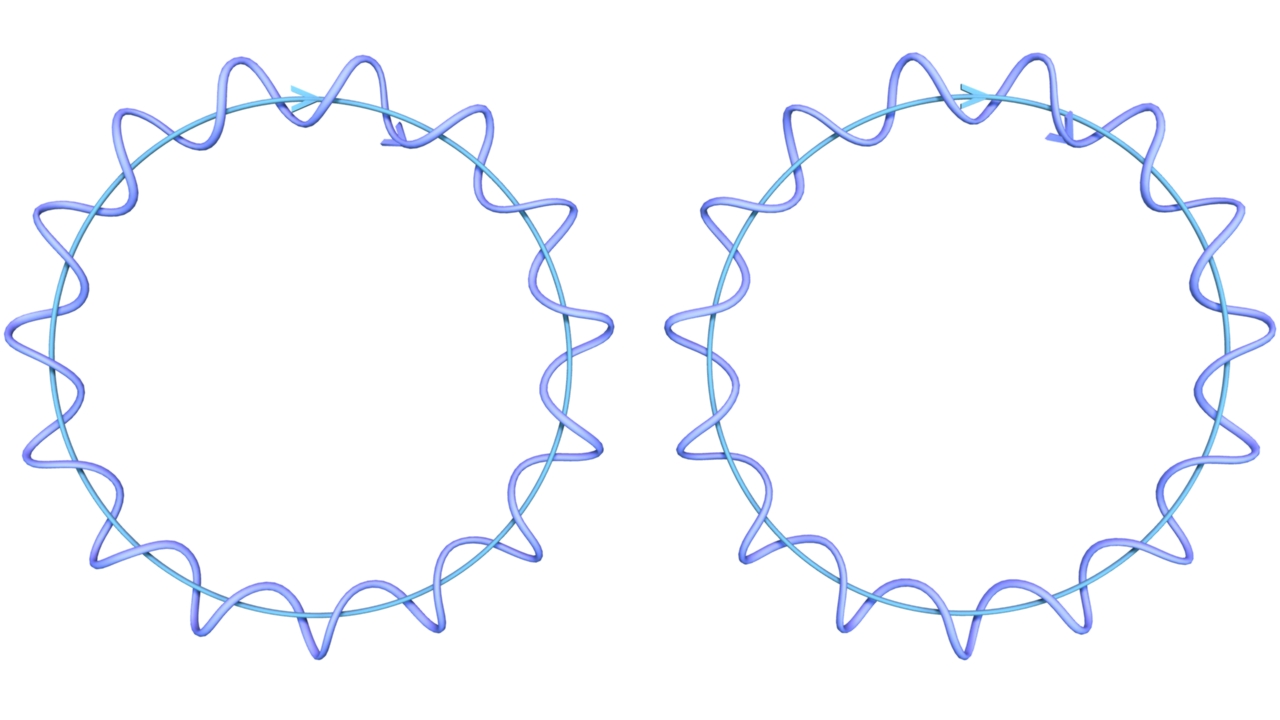}
  \caption{Computing the linking number depends only on the crossings between the knots but not on their knot types.}
  \label{linking2}}
\end{figure}

\subsection{Characterizations of the linking number}
In this section we give various geometric and combinatorial characterizations for the linking number and show the equivalence between them. The definitions of the linking number will be used in later sections.

First, consider a knot $K$ in $S^3$, and then its complement $S^3\backslash K$.  A quick application of Alexander Duality tells us that the first homology group $H_1(S^3\backslash K) \cong \mathbb{Z}$, and is thus generated by a single element $[\eta ]$ (see Figure \ref{se_alg}). 

\begin{figure}[h]
  \centering
   {\includegraphics[scale=0.13]{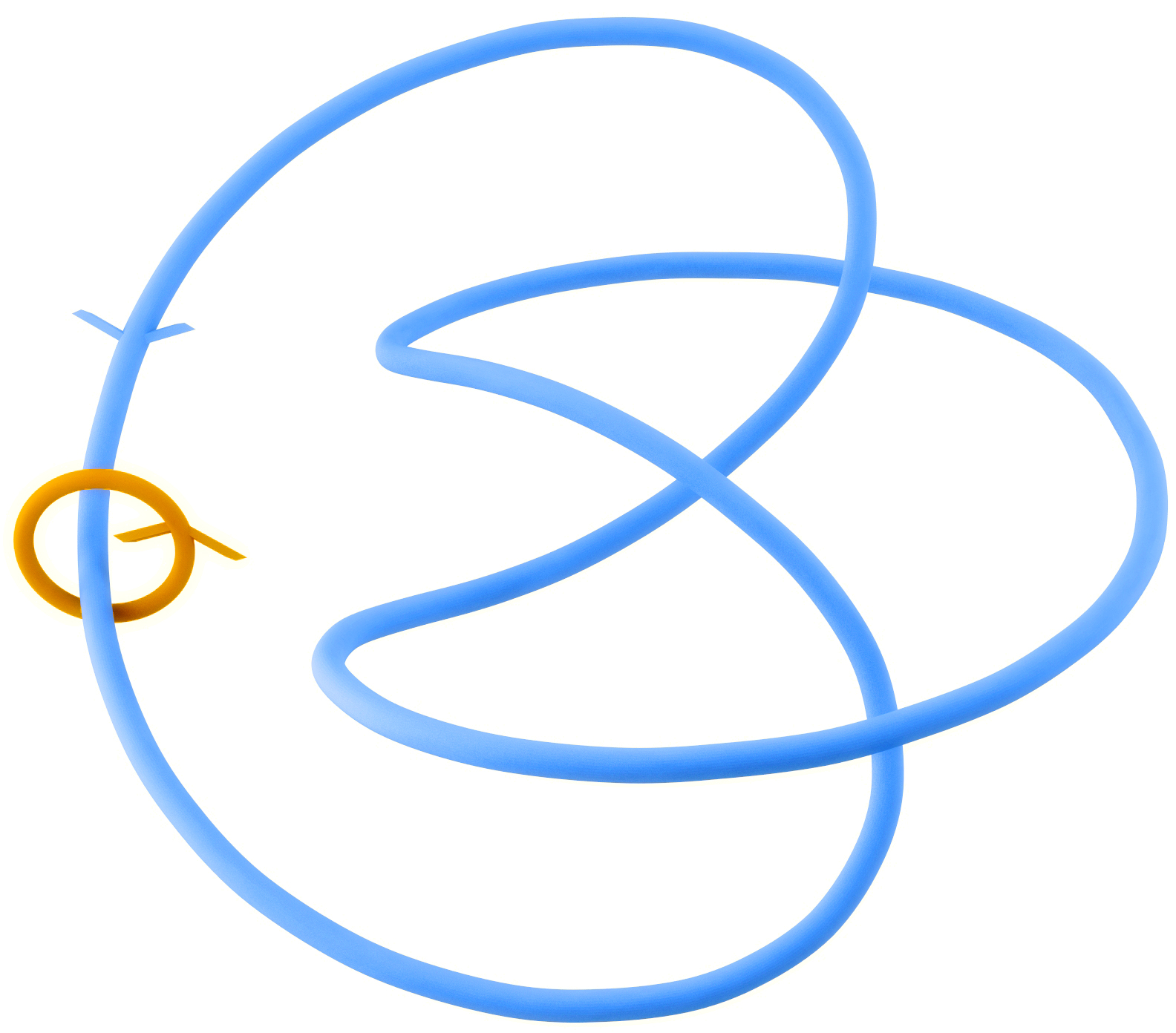}
         \put(-210,90){$\eta$}
 \put(-110,10){$K$}
  \caption{The curve $\eta$ generates the group $H_1(S^3\backslash K)$. As there are two such generators, we usually choose $\eta$ to be the curve with $lk(\eta,K)=1$.}
  \label{se_alg}}
\end{figure}

Let $J$ and $K$ be two disjoint oriented knots in $S^{3}.$ The curve $J$ can be
regarded as a loop in $S^{3}\backslash K$, so it represents an element of the first homology $H_{1}(S^{3}\backslash K)$. This group is generated by the curve $[\eta ]$ (Figure \ref{se_alg}), so write $[J]\in H_{1}(S^{3}\backslash K)$ in terms of the generator $\ [\eta].$ Namely, $[J]_{S^{3}\backslash K}=s[\eta ]$ for some $s\in\mathbb{Z}$. Theorem \ref{chara} below shows that this integer $s$ is equal to $lk(J,K)$ (see Figure \ref{se_alg}).


We now give another characterization of the linking number. A \textit{{Seifert surface}} of a knot is a compact, connected, orientable surface whose boundary is the knot. See Figure ~\ref{seifert} and also this \href{http://www.youtube.com/watch?v=px3Gq_gvvac}{Seifert surface movie}) \footnote{ The Seifert surface movie can be viewed by visiting http://www.youtube.com/watch?v=px3Gq\_gvvac.}.
When the knot $K$ is oriented, we will always assume that the Seifert surface $S$ of $K$ is oriented in a way such that $\partial S=K$.

\begin{figure}[htb]
  \centering
   {\includegraphics[scale=0.2]{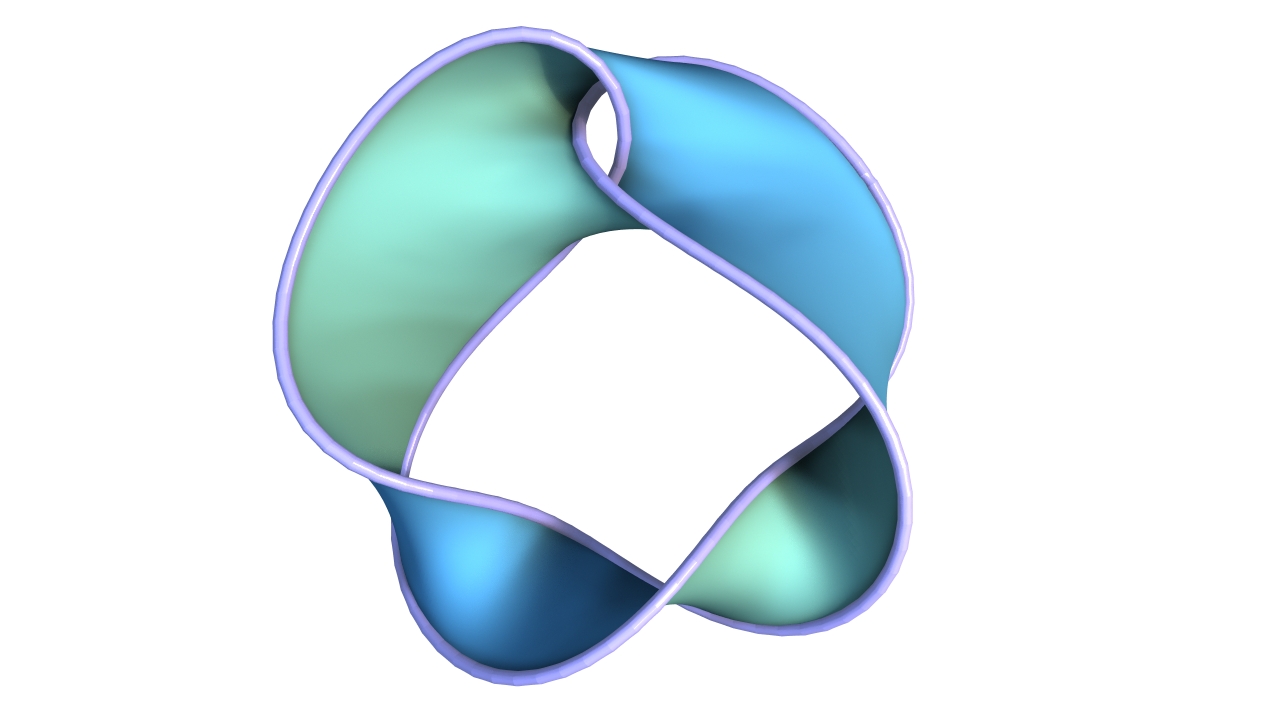}
  \caption{A Seifert Surface for the figure-$8$ knot.}
  \label{seifert}}
\end{figure}

For an orientable surface $S$ with an oriented boundary we need to distinguish between the two sides of $S$. We define the positive side to be the side that its oriented boundary runs counterclockwise as it is seen from it. We denote this side by $S^+$. The side $S^-$ is defined similarly (see Figure \ref{se_surface}). 
\begin{figure}[h]
  \centering
   {\includegraphics[scale=0.1]{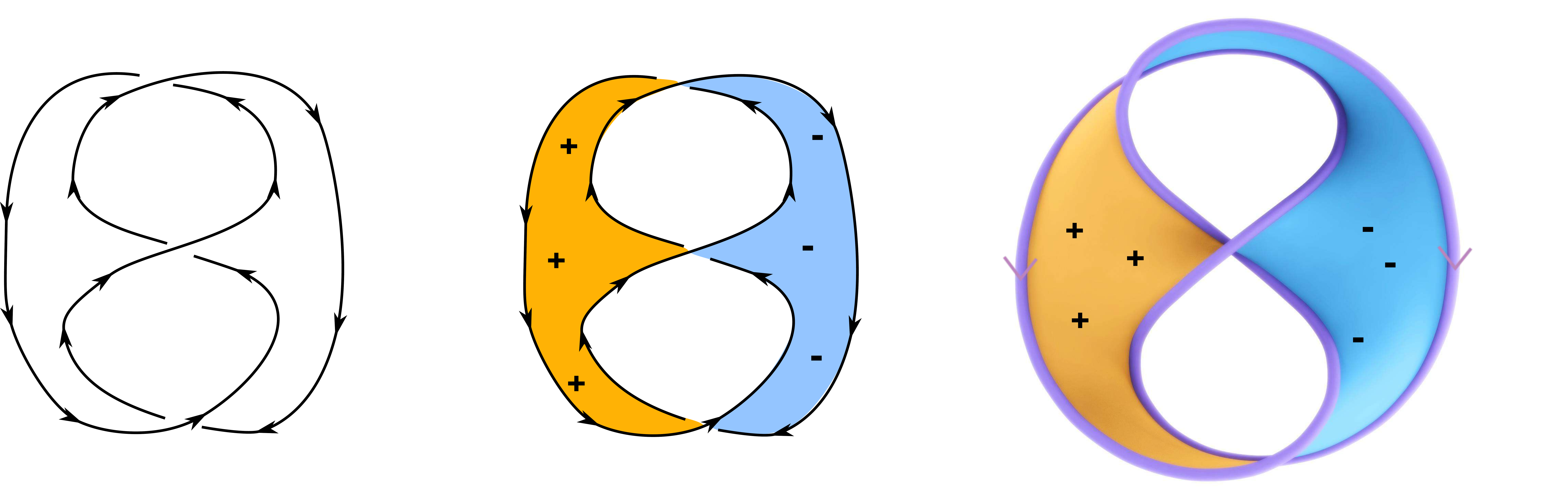}
 \put(-330,-15){$(A)$}
 \put(-220,-15){$(B)$}
  \put(-90,-15){$(C)$}
  \caption{(A): an oriented knot $K$. Figures (B) and (C) shows a Seifert surface of the oriented knot $K$. The Seifert surface  determines two sides $S^+$ and $S^-$.}
  \label{se_surface}}
\end{figure}

\begin{definition}
Let $S$ be a Seifert surface for an oriented knot $K$ in $S^3$. Let $J$ be an oriented knot in $S^3$ that is disjoint from $K$. A \textbf{\textit{positive (resp. negative) intersection}} of $S$ with $J$ is a transverse intersection of $S$ with $J$ such that the oriented curve $J$ passes from $S^-$ to $S^+$ (resp. $S^{+}$ to $S^{-}$). Assign weights $+1$ and $-1$ respectively to the positive intersections and negative intersections of $S$ and $J$.  The \textit{\textbf{intersection number}} of $S$ and $J$, denoted $S \cdot J$, is the sum of the weights of all transverse intersections. The following theorem will prove that the $S \cdot J$ is equal to $lk(J,K)$. 


\end{definition}
\begin{theorem}
\label{chara}
Let $J$ and $K$ be disjoint oriented knots in $S^3$. Let $S$ and $S^{\prime}$ be a Seifert surfaces that bounds $K$ and $J$ respectively. Then
\begin{itemize}
\item[(1)] Suppose that $[\eta]$ generates $H_{1}(S^{3}\backslash K)\cong\mathbb{Z}$, where $[\eta]$ is represented by a curve $\eta$ such that $lk(K,\eta)=1$.  Then if   $[J]=s[\eta]$ for some $s\in \mathbb{Z}$, we have that  $lk(J,K)=s$.
\item[(2)] $lk(J,K)=J \cdot S= K \cdot S^{\prime}$
\end{itemize}
\end{theorem}
\begin{proof}
\begin{enumerate}
\item Suppose that $[J]=s[\eta]$ for $s\in \mathbb{Z}$. Turn each positive crossing of $J$ under $K$ into an overcrossing by replacing $J$ with the connected sum $J \# (-\eta)$, remembering that $-\eta$ denotes the curve $\eta$ with its orientation reversed. Turn each negative crossing of $J$ under $K$ into an overcrossing by replacing $J$ with the connected sum $J \# \eta$. Doing this for all undercrossings of $J$ with $K$ gives us two knots $K$ and $J \# (-s\eta)$ that can be separated by a 2-sphere.  As such, we can manipulate them via ambient isotopy in such a way that they share no crossings in a planar diagram, demonstrating that $lk(K,J \# (-s\eta))=0$. Hence $lk(K,J \# (-s\eta))=lk(K,J)-s \; lk(K,\eta)=lk(K,J)-s=0$ which yields the result. See Figure \ref{cancel1}.

\begin{figure}[h]
  \centering
   {\includegraphics[scale=0.2]{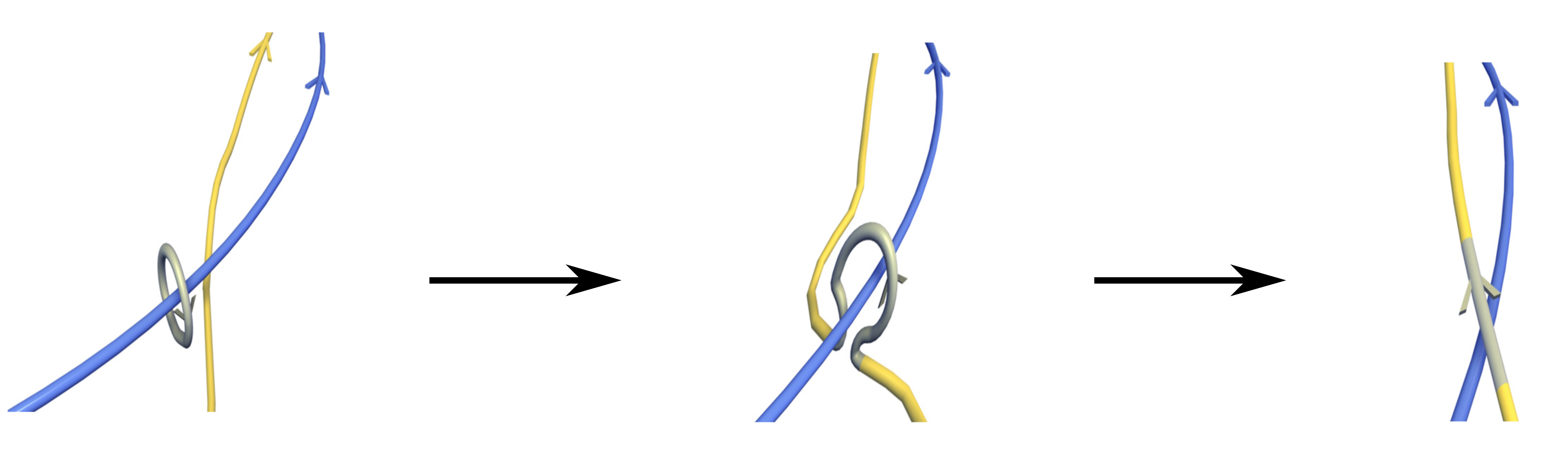}
      \put(-350,-20){$(A)$}
    \put(-190,-20){$(B)$}
     \put(-20,-20){$(C)$}
  \caption{Unlinking two knots $K$ and $J$ locally at a crossing corresponds to
taking the connected sum of $J$ and the curve $\pm[\eta]$ where $[\eta]$ is a curve with $lk(K,\eta) =1$.}
  \label{cancel1}}
\end{figure}


\item 
It is sufficient to show $lk(K,J)=K\cdot S^{\prime}$. The equality $lk(K,J)=J\cdot S$ follows by the symmetry of the linking number (Theorem \ref{linkingnumber} part (3)).
Consider the curve $J$ as an element in $H_{1}(S^{3}\backslash K)$ which is generated by $[\eta ]$. Write $[J]_{S^{3}\backslash K}=s[\eta ]$ for some $s\in \mathbb{Z}$. We know by part (1) that $s= lk(J,K)$. The result follows if we show that $s=K. S^\prime$. Now let set $n=K. S^\prime$. 
Suppose $K$ intersects $S^\prime$ positively at a point $p$ as indicated in the Figure \ref{cancel} (A). Inspect the connected sum $J\# -\eta$ and notice that it cancels the intersection point between $S^\prime$ and $K$ around $p$. Similarly, when $K$ intersects $S^{\prime}$ negatively, the connected sum $J\# +\eta$ cancels one intersection between the surface $S^{\prime}$ and $K$. Now, let $-n \eta$ be the curve obtained from $n$ copies of $-\eta$ when $n>0$ or $-n$ copies of $\eta$ when $n<0$.

By our earlier observation, the homology element $[J\#-n\eta]_{S^{3}\backslash K}$ bounds a surface that does not intersects the knot $K$. Hence homology element $[J\#-n\eta]_{S^{3}\backslash K}=0$ or $[J]=n[\eta]$. The result follows.


\begin{figure}[h]
  \centering
   {\includegraphics[scale=0.2]{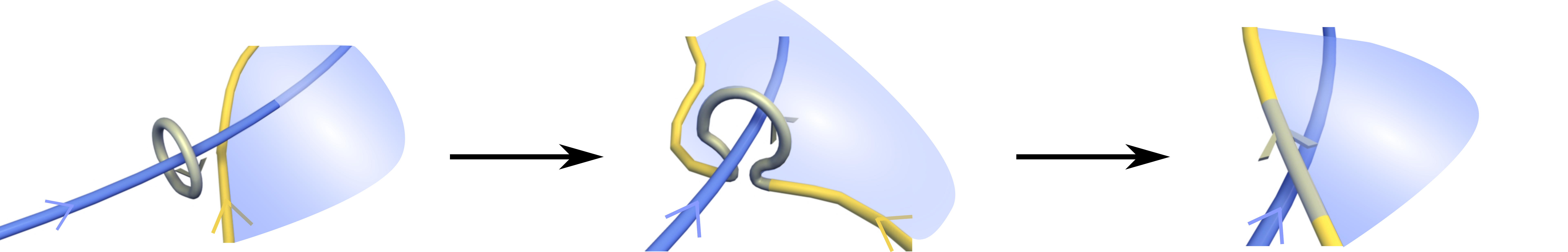}
      \put(-400,-20){$(A)$}
    \put(-240,-20){$(B)$}
     \put(-80,-20){$(C)$}
  \caption{The blue curve represents the knot $K$, the yellow curve represents the knot $J$, and the grey curve represents the curve $\eta$. The surface $S^\prime$ is the surface colored in blue and it is the Seifert surface of $J$. This figure shows that reducing the number of intersections between $S^\prime$ and $K$ by $1$ corresponds to
taking the connected sum of $K$ and the curve $\pm[\eta]$.}
  \label{cancel}}
\end{figure} 

\end{enumerate}       
\end{proof}



If we orient the knot $K$ then for a framed knot $(K,V)$ we can define explicitly what we mean by the framing integer $n$ that describes the number of times the vector field twists around $K$, as follows:
\begin{definition}
\label{pushoff}
Let $(K,V)$ be a framed knot. The self-linking number is given by $lk(K,K^{\prime})$, where $K^{\prime}$ is an oriented knot formed by a small shift of $K$ in the direction of the framing vector field and oriented parallel to the knot $K$ (see Figure \ref{selflinking}).
\end{definition}

\begin{figure}[htb]
  \centering
   {\includegraphics[scale=0.2]{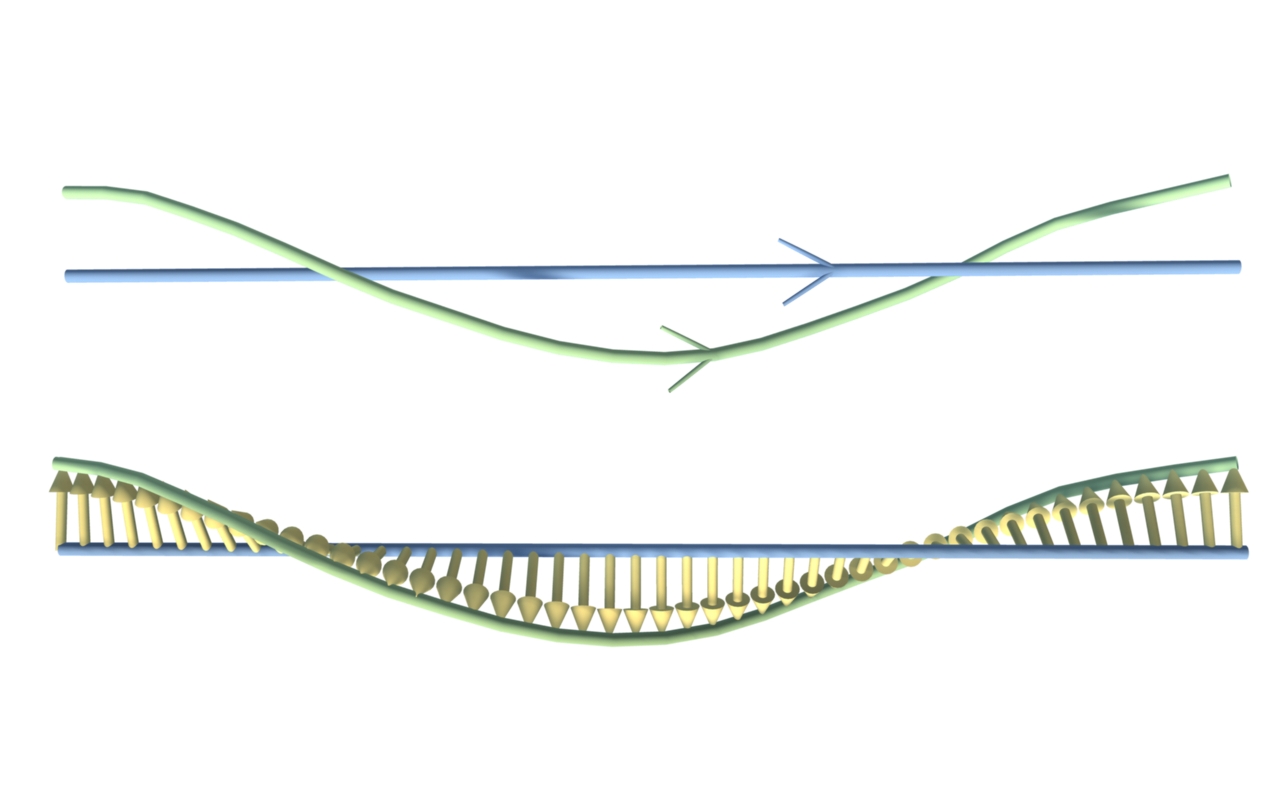}
  \caption{A pushoff of a knot by a framing.}
  \label{selflinking}}
\end{figure}

Note that the self-linking number of a framed knot is independent of the
orientation we choose for $K$, since at every crossing of $-K$ the orientation of both arcs is reversed, leaving the sign unchanged.  Note also that the self-linking number is the same if $K$ is shifted in the direction opposite to the framing.

The reason we have introduced this concept is that the self-linking number of a framed knot $(K,V)$ is equal to the framing integer that determines, or is determined by $(K,V)$. This is evident by observing Figure \ref{selflinking} and noticing that locally, the vector field winds $\pm 1$ around the knot if and only if the pushoff $K^{\prime}$ contributes $\pm1$ to the self-linking number.  Note that the definition of a framed knot $(K,V)$ is independent of
the choice of orientation of the knot $K$. On the other hand we have just shown that the self-linking number is independent of orientation we choose for the knot $K$ so defining this number to be the framing integer matches with our original definition of the framing. Hence we will assume in what follows that these two concepts, the self-linking number and the framing integer, are the same and we will use both terms interchangeably.
The framing with self-linking number $n$ will be called
the \textbf{\textit{$n$-framing}} and a knot with the $n$-framing will be referred to as \textbf{\textit{$n$%
-framed}}. Hence we can define a framed knot in $S^{3}$ to be $(K,n)$ where $K$
is a knot in $S^{3}$ and $n$ is an integer.  It will be useful in practice to have a standard way to choose a framing, given a knot diagram.

\begin{definition}
The \textbf{\textit{blackboard framing}},
defined for a plane knot projection, is given by a nonzero vector field that is everywhere
parallel to the projection plane. See Figure \ref{framedtrefoil}.
\end{definition}

\begin{figure}[htb]
  \centering
   {\includegraphics[scale=0.2]{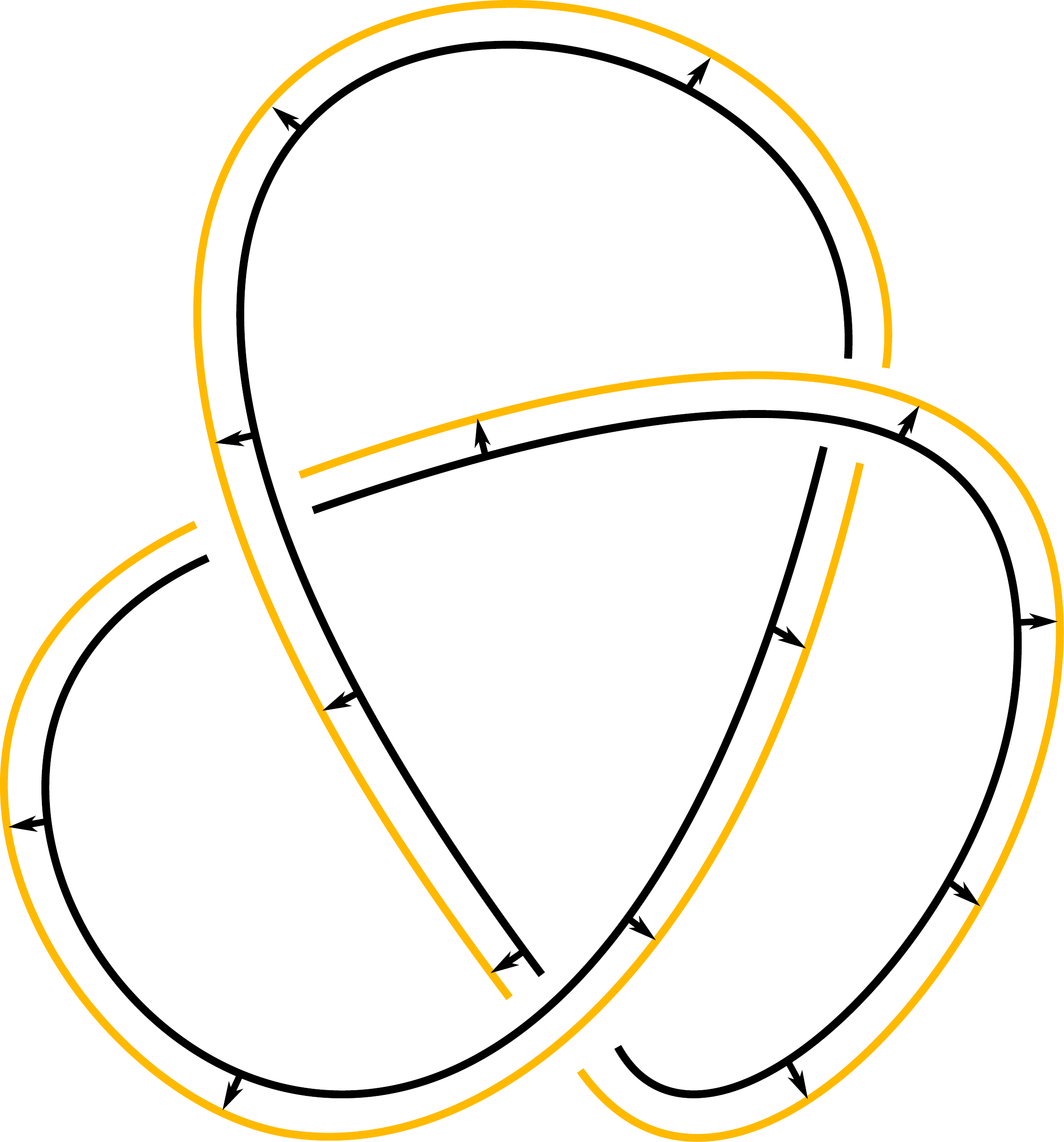}
  \caption{ A trefoil diagram with its blackboard framing.}
  \label{framedtrefoil}}
\end{figure}

The reason for calling this the blackboard framing is clear once we attempt to draw it: we simply choose a point on the knot, move transversely to the knot on one side (it doesn't matter which!) and then follow the knot, staying on the same side of the arcs until the chalk returns to the original pushoff.  In this way, if we visualize the framed knot as a ribbon, it will lie flat on the blackboard.

The blackboard framing is also related to the notion of \textit{writhe} of a knot. 

\begin{definition}
The \textbf{\textit{writhe}} of a knot diagram is the sum of the signs of every crossing in the diagram.
\end{definition}
Notice that since both possible choices of orientations give the same sign at each crossing then the writhe does not depend on the orientation of the knot.  Note also that the writhe is invariant under $\Omega_{2}$ and $\Omega_{3}$ but not invariant under the move $\Omega_{1}$. The notions of the writhe of knot diagram and the self-linking of a framed knot given by a diagram with a blackboard framing are related as we will show shortly.

The blackboard framing for a knot diagram $D$ of $K$ corresponds to one particular framing $%
n_{0}$ of $K$.  This leads to a natural question: can we obtain a ``framed knot diagram"
corresponding to each of the possible framings for $K$? In other words, can
we always represent a framed knot $(K,n)$ by a knot diagram with the blackboard framing? The answer is \emph{yes}.
In order to see this we need to see the effect of
the Reidemeister moves $\Omega _{1}$, $\Omega _{2}$, and $\Omega _{3}$ on
the blackboard framing of a fixed knot diagram $K$.  Notably, only $\Omega_{1}$ changes the blackboard framing, by exactly $\pm 1$.  By applying an appropriate number of the moves $\Omega_{1}$ we can thus find a diagram of the knot with the desired framing being the blackboard framing.

What we have also discovered is that for framed knots (with blackboard framed diagrams) the Reidemeister theorem does not hold immediately because the move $\Omega_{1}$ changes the
blackboard framing. Luckily there is an analogous theorem, which will follow directly once we prove the following proposition.

\begin{proposition}
The self-linking number of a framed knot given by a diagram with
blackboard framing is equal to the writhe of the diagram.
\end{proposition}

\begin{proof}
In the case of blackboard framing, the only crossings of $K$ with its pushoff $K^{\prime
}$ occur near the crossing points of $K$. The neighborhood of each crossing
point looks like

\begin{figure}[htb]
  \centering
   {\includegraphics[scale=0.85]{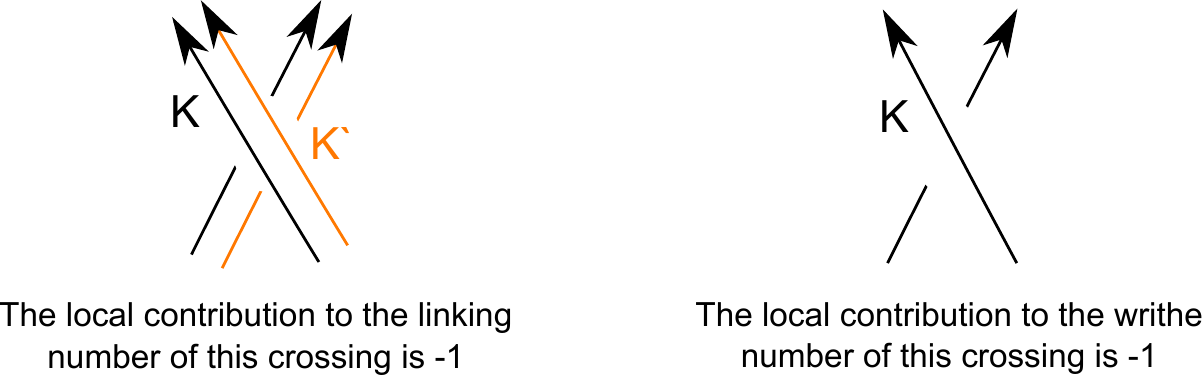}
  \caption{The writhe of a diagram is the same as the self-linking number.}
  \label{the writhe and selflinking}}
\end{figure}

There are two crossings of $K$ with $K^{\prime}$, each with the same sign as the crossing of $K$. The claim follows directly from the
definition for the linking number in $S^{3}$, and we now see some of the motivation for defining the writhe to be the total sum, whilst the linking number is one half of the sum of the crossings.
\end{proof}
Now we give the figure of modified Reidemeister move 1, which we will use in the next theorem.
\begin{figure}[htb]
  \centering
   {\includegraphics[scale=0.38]{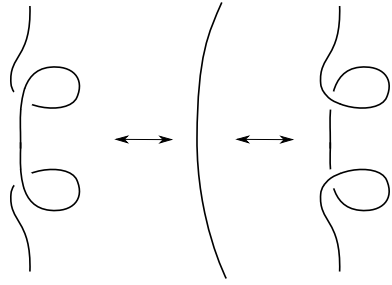}
  \caption{Modified Reidemeister 1 move $F\Omega _{1}$.}
  \label{modified r1 move}}
\end{figure}

\begin{theorem}
Two knot diagrams with blackboard framing $D_{1}$ and $D_{2}$ represent equivalent framed knots
if and only if $D_{1}$ can be transformed into $D_{2}$ by a sequence of
plane isotopies and local moves of the three types $F\Omega _{1}$, $\Omega _{2}$%
, and $\Omega _{3}$, where $F\Omega _{1}$ is given by the figure~\ref{modified r1 move} and $\Omega _{2}$ and $\Omega _{3}$ are the usual Reidemeister moves.
\end{theorem}

\begin{proof}
Suppose first that the diagrams represent equivalent framed knots.
The associated knots $K_{1}$ and $K_{2}$ are isotopic, and thus the 
\emph{standard} Reidemeister theorem tells us that the diagrams are related by a sequence of plane isotopies and the moves $\Omega _{1}$, $\Omega _{2}$%
, and $\Omega _{3}$.  Note that by the above proposition, $D_{1}$ and $D_{2}$ both have the same writhe.  We know that writhe is invariant under plane isotopies and the moves $\Omega _{2}$ and $\Omega _{3}$, and moreover that every move $\Omega_{1}$ changes the writhe by exactly $\pm 1$, with the sign depending on the direction of the kink.  Thus, there must be an even number of right-kinks and left-kinks in the sequence of moves connecting $D_{1}$ to $D_{2}$.  By a sequence of plane isotopies, $\Omega_{2}$ and $\Omega_{3}$ moves any kink can be moved anywhere along the knot.  We can then pair them so that we get a set of moves of the form $F\Omega _{1}$, and this direction of the statement is proved. For the other direction, we need simply to note that the modified move $F\Omega_{1}$ doesn't change the writhe of a diagram, and is a combination of traditional Reidemeister moves.  Hence two diagrams being related by a sequence of these moves means that the corresponding knots are isotopic, and they have the same framing.
\end{proof}

The previous results can be summarized in the following statement.  For every framed knot $(K,n)$ we can find a plane knot diagram that
represents that framed knot. This plane diagram is unique up to modified Reidemeister moves and plane isotopy.

In the next section we introduce two important curves that are naturally related to a framed knot in $S^3$. 
\section{The longitude and the meridian}
In this section provide various characterization of two important curves that are related to the framing of a knot. These curves provide another homological characterization of the framing of a knot. Furthermore we relate these curves to the self-linking number we introduced earlier. This definition of the framing plays an essential role when one defines a surgery on $3$-manifold. 
 
Before we introduce these curves and their relationship to the framing of knot we need to discuss the homology and the homotopy groups of the torus.

\subsection{Curves on the torus}
In this subsection we give a discussion of closed curves on the torus up to three equivalence relations: homology, homotopy and ambient isotopy.\\
 Let $S$ be an arbitrary surface. Let $f_{i}:[0,1]\longrightarrow S$ for $i=1,2$ be two loops ($f_i(0)=f_i(1)$) on the surface. It is easy to prove the following facts:

\begin{enumerate}

\item If $f_1$ is homotopic to $f_2$ then $f_1$ is homologous to $f_2$.
\item Suppose that $f_1$ and $f_2$ are embeddings. If $f_1$ is ambient isotopic to $f_2$ then $f_1$ is  homotopic to $f_2$ in $S$.
\end{enumerate}
For a generic surface $S$ the inverse of the statements (1) and (2) is not true in general. On the torus however, the inverse directions hold in special cases. We discuss this in the following. 
\subsubsection{Homology and homotopy of the torus}

We give a quick discussion on the first homology and homotopy groups of the tours. See the first chapter of \cite{Rolfsen} for more details.

The fundamental group of the torus is $\pi _{1}(T^{2})=\pi _{1}(S^{1}\times S^{1})\simeq \mathbb{Z}\bigoplus \mathbb{Z}$. In what follows we will define a particular isomorphism between the fundamental group of $T^2$ and $\mathbb{Z}\bigoplus \mathbb{Z}$. Hence, specify coordinates for torus $T^2$ by  $T^{2}=S^{1}\times S^{1},$ where we identify $S^{1}$ with  the unit
complex numbers. Then any point on $T^{2}$ has coordinates $(e^{2 \pi i\theta
},e^{2 \pi i\phi })$ where $0\leq \theta, \phi\leq 1$. Furthermore, choose the counterclockwise orientation on $S^1$ and in this way any map $f$:$S^{1}\rightarrow T^{2}$ may be
regarded as an element of $\pi _{1}(T^{2})$. In particular, consider the maps $l:S^{1} \rightarrow T^{2}$ and $m:S^{1} \rightarrow T^{2}$ given by
\begin{equation*}
m(e^{2 \pi i\theta})=(1,e^{2 \pi i\theta})
\end{equation*}
\begin{equation*}
l(e^{2 \pi i\theta})=(e^{2 \pi i\theta}, 1)
\end{equation*}

\noindent where $0\leq \theta \leq 1.$
These two maps represent the two generators
of $\pi _{1}(T^{2})$. See Figure \ref{merdian_long}.

\begin{figure}[h]
  \centering
   {\includegraphics[scale=0.12]{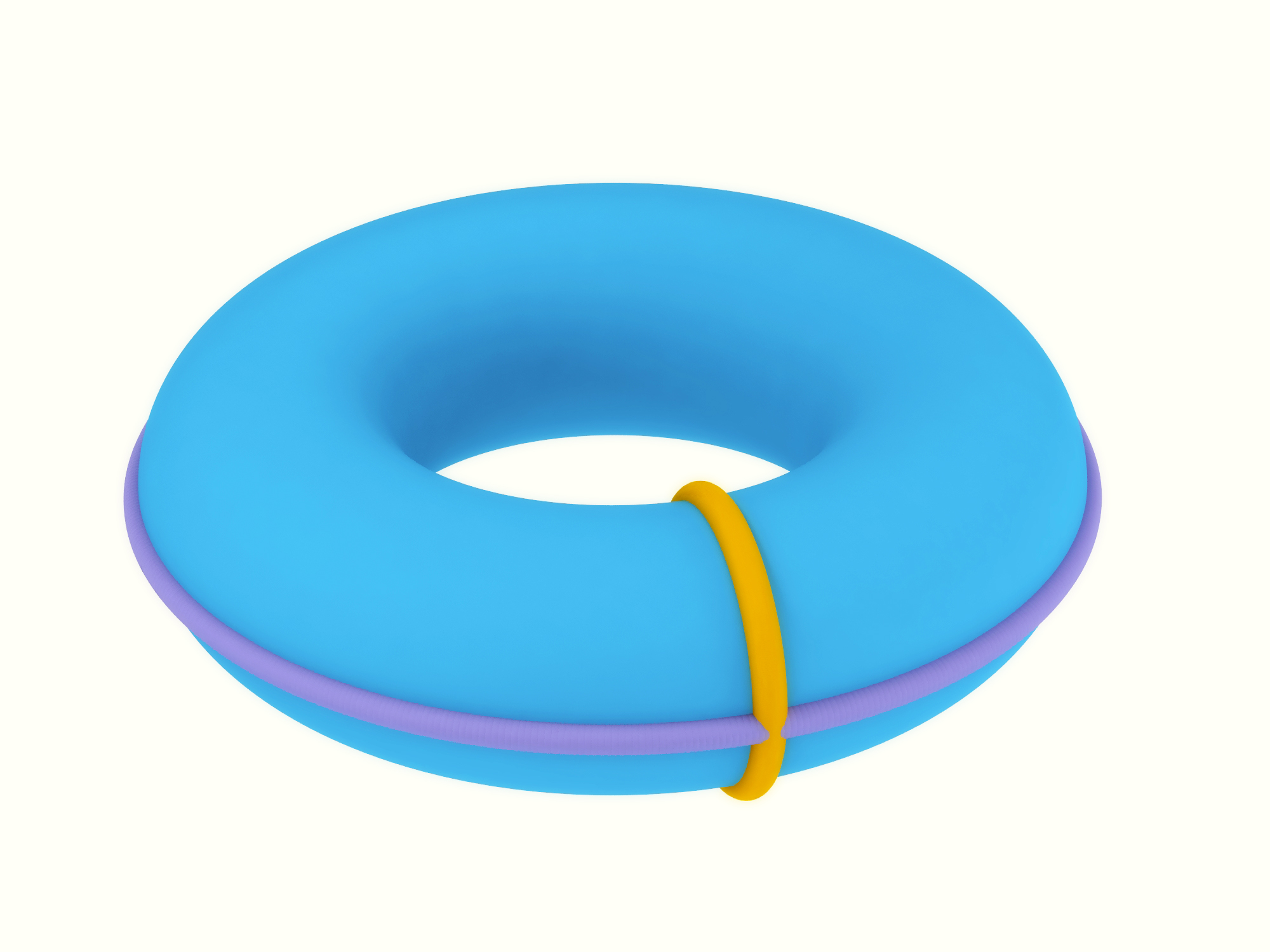}
 \put(-114,70){$m$}
  \put(-35,67){$l$}
  \caption{The maps $m$ and $l$ on the standardly embedded torus in $\mathbb{R}^3$.}
  \label{merdian_long}}
\end{figure}

We define an isomorphism between $\pi _{1}(T^{2})$ and $\mathbb{Z}\oplus \mathbb{Z}$ by sending $m$ to $(0,1)$ and $l$ to $(1,0)$. Hence, any class in $\pi _{1}(T^{2})$ can be represented by $(n,m)$
where $n,m\in \mathbb{Z}.$ We neglect the base point here because $T^{2}$ is path connected. Since $\pi
_{1}(T^{2}) $ is abelian we also know that the groups $\pi _{1}(T^{2})$ and $H_{1}(T^{2})$ are isomorphic. In other words two closed curves in $T^{2}$ are homotopic if and only they are homologous.

The curves $m$ and $l$ are easily defined in the case when $T^2$ has the above parameterization. However, this definition is more involved when the torus $T^2$ is embedded in $S^3$. We study these curves in \ref{curves on torus}.

\subsubsection{Knots on the Torus}
Given a closed curve $C$ on the torus $T^{2}$ that represents a class in $\pi _{1}(T^{2})$. Does there exist a simple closed curve $C^{\prime}$ in  $T^{2}$ that is homotopic to $C$? In other words, when can we represent a homotopy class in $\pi _{1}(T^{2})$ by an embedding in $T^2$? The answer in general is \emph{no}. For instance we cannot find a simple closed curve that represents the homotopy class $(2,0)$. On the other hand one can find a simple closed curve that represents the class $(2,3)$. The following two theorems answer this question.

\begin{theorem}(page 19 in \cite{Rolfsen})
\label{mainmain}
Let $c$ be a curve in $T^2$ with a homotopy class $(a,b)$ in $\pi _{1}(T^{2})$. The curve $c$ can be represented by an embedding $S^{1}\rightarrow T^{2}$ if and only if either of the integers $a,b$ are coprime (that is $g.c.d.(a,b)=1$) or one of them is zero and the other is $\pm 1$, or $a=b=0$. 
\end{theorem}


This theorem is useful when we want to know if an embedded curve representation of a certain homotopy class exist on the torus. As Theorem~\ref{mainmain} asserts, such an embedding exists if and only if the pair $(a,b)$ which completely characterizes the homotopy type of the class, satisfies the conditions mentioned in the theorem.  The next theorem also relates the homotopy classes and isotopy classes of curves on $T^2$.  

\begin{theorem}\cite{Rolfsen}
\label{ambient}
If two closed curves without self-intersection on the torus $T^2$ are homotopic then they are ambient isotopic. 
\end{theorem}
In summary, if we are given a homotopy class $(a,b)$ in $T^2$ such that $a=b=0$, or one of them is zero and the other is $\pm 1$ or $g.c.d.(a,b)=1$, then we can represent the class $(a,b)$ by a simple closed curve in $T^2$. Moreover, given two such representations of this homotopy class $(a,b)$, without self-intersection, then one can find an ambient isotopy on $T^2$ that takes the first representation to the second one. The proofs of the previous two theorems are omitted and the interested reader is referred to first chapter \cite{Rolfsen} for details. See also chapter $6$ in \cite{prasolov1997knots}.

\subsection{The longitude and the meridian of an embedded torus in $S^3$}
\label{curves on torus}
Let $K$ be an oriented knot in $S^{3}$. Let $N\subset S^{3}$ be a tubular
neighborhood around $K,$ i.e. a solid torus embedded in the 3-sphere whose core
is the knot $K$. It is easiest to think of $N$ as just a thickening of $K$.  Let $X$ denotes the closure of $S^{3}-N.$ We assume that $
N $ is embedded in $S^{3}$ so that $X$ is a manifold. In this case it clear
that $\partial X=\partial N=T^{2}.$ The Mayer-Vietoris exact sequence for $
S^{3}=N\cup X$ with $N\cap X=T^{2}$ reads
\begin{equation}
\label{exact}
H_{2}(S^{3}) \rightarrow H_{1}(X\cap N)\rightarrow H_{1}(X)\oplus H_{1}(N)\rightarrow H_{1}(S^{3})
\end{equation}
From basic homology theory we know that $H_{1}(S^3)=H_{2}(S^3)=0$. Moreover, since $X\cap N$ is homeomoprhic to a torus then we know from the previous section that $H_{1}(X\cap N)=\mathbb{Z}\oplus \mathbb{Z}$. Finally, since $N$ is homotopic to the knot $K$ then $ H_{1}(N)$ is isomorphic to $\mathbb{Z}$ hence we can write equation (\ref{exact}) as follows:
\begin{equation}
\label{this}
\indent \indent  0 \indent \rightarrow \indent \mathbb{Z}\oplus \mathbb{Z} \indent \rightarrow \indent H_{1}(X)\oplus \mathbb{Z} \indent \rightarrow \indent 0
\end{equation}

The sequence (\ref{this}) is exact. Hence the middle map is an isomorphism and thus $H_{1}(X)$ is isomorphic to $\mathbb{Z}$. We will choose a specific isomorphism between $H_{1}(X)$ and $\mathbb{Z}$
later in this section. By the Mayer-Vietoris Theorem the isomorphism 
\begin{equation}
\label{main map}
 i_{\ast }\oplus j_{\ast }: H_{1}(X\cap N)
\rightarrow H_{1}(X)\oplus H_{1}(N)
\end{equation}

\noindent is given explicitly by $i_{\ast }\oplus j_{\ast }([c])=(i_{\ast }[c],j_{\ast }[c])$
where $i:X\cap N\hookrightarrow X$ and $j:X\cap N\hookrightarrow N$ are the
inclusion maps. Note that the map $i_{\ast }$ pushes curves on the surface $
X\cap N$ into the knot exterior $X,$ and similarly the map $j_{\ast }$
pushes curves on $X\cap N$ into the solid torus $N$.

\subsubsection{The meridian}

Recall that $N$ is homemorphic to a solid torus, and its
boundary $X\cap N$ is homeomorphic to a torus. We then know that $H_{1}(X\cap N)\cong \pi_{1}(T^{2})$ is
generated by two curves, and by Theorem~\ref{thm5.2}, each of which can be chosen to be simple and closed.
One of these curves, denoted $\eta,$ can be chosen to encircle the knot $K$
and bound a disk in $N.$ We can further choose the orientation on $\eta$ so that $lk(K,\eta)=1$. Because $\eta$
bounds a disk in $N$ it is null-homologous in $N$ and hence $j_{\ast}[\eta]=0$ in $H_{1}(N)$. Now, the meridian $[\eta]$ represents a generator of $H_{1}(X\cap N),$
and thus any isomorphism $i_{\ast}\oplus j_{\ast}$ must map it to a generator in
$H_{1}(X)\oplus H_{1}(N).$  Using our explicit definition of the map, we see that $i_{\ast}\oplus j_{\ast}([\eta])=$ $(i_{\ast}[\eta],j_{\ast}[\eta])=(i_{\ast}[\eta],0)$. In other words $i_{\ast}[\eta]$ generates the group $H_{1}(X)$. We use this generator to give a specific isomorphism $H_{1}(X)\rightarrow \mathbb{Z}$ defined by sending $i_{\ast}[\eta]$ to 1. We will refer to the homology class $i_{\ast}[\eta]$ in $H_1(X)$ by $[\eta]$. 

\subsubsection{The preferred longitude}
Note that the solid torus $N$ is homotopy equivalent to its core $K$, allowing us to represent the generator of $H_{1}(N)$ by the oriented knot itself.  We then fix an isomorphism
$H_{1}(N) \rightarrow \mathbb{Z}$ that maps $[K]$ to $1$. In the previous section we defined the isomorphism $H_{1}(X)\rightarrow \mathbb{Z}$ that sends the homology class of the curve $\eta$ to $1$ in $\mathbb{Z}$. Using these two isomorphisms we can construct a specific isomorphism between $H_{1}(X)\oplus H_{1}(N)$ and $\mathbb{Z}\oplus \mathbb{Z}$ defined by $(0,[K]) \rightarrow (0,1)$ and $([\eta],0) \rightarrow
(1,0). $ We will assume this identification from now on. 
Since $i_{\ast }\oplus j_{\ast }$ is an isomorphism, there exists a
unique element $[\gamma]$ in $H_{1}(X\cap N)$ that maps to $(0,1)$.  Since the class $(0,1)$ is a generator in $H_{1}(X)\oplus
H_{1}(N)$, the image element $[\gamma]$ under the isomophism $(i_{\ast }\oplus j_{\ast })^{-1}$ must also be a generator in $H_{1}(X\cap N)$. Hence, by Theorem \ref{mainmain}, we can represent the class $[\gamma]$ by a simple closed curve (that will also be denoted $\gamma$) on $T^{2}\cong X\cap N$.  We interpret $\gamma$ as follows: $i_{\ast }\oplus j_{\ast }([\gamma
])=(i_{\ast}[\gamma],j_{\ast}[\gamma])=(0,1)\in H_{1}(X)\oplus H_{1}(N)$ means that $[\gamma]_{X}$ is null-homologous in $H_{1}(X)$ and $[\gamma]_{N}=[K]$ in $H_{1}(N)$.

\begin{remark}
If we consider a simple closed curve as a representative of the homology class $[\gamma]_{X}$, and denote it by $\gamma$, then this curve can be seen to be obtained by an ambient isotopy of the knot $K$ inside $N$. We can choose the curves that connect the beginning of the ambient isotopy, namely $K$, to the end of it, $\gamma$, to be a collection of simple closed embedded curves in $N$ and hence these curves define a ribbon tangle or a framed knot with boundary being the union of knot $K$ and the curve $\gamma$. 

\end{remark}
We give some facts about the meridian and the preferred longitude in the following definition.

\begin{definition}\label{Merid-long}
Let $K$ be an oriented knot in (oriented) $S^{3}$ with solid torus
neighborhood $N$. A \textit{\textbf{meridian}} $\eta $ of $K$ is a non-separating simple curve
in $\partial N$ that bounds a disk in $N$. A \textit{\textbf{preferred longitude}} $\gamma $
of $K$ is a simple closed curve in $\partial N$ that is homologous to $K$ in
$N$ and null-homologous in the exterior of $K$. 
\end{definition}

The previous discussion about meridian and longitude implies immediately the following theorem.

\begin{theorem}
Let $K$ be an oriented knot in (oriented) $S^{3}$ with solid torus
neighborhood $N$. Then the following facts hold:
\begin{itemize}
\item The meridian $\eta$ is a simple closed
curve that generates the kernel of the homomorphism $H_{1}(X\cap
N)\rightarrow H_{1}(N)$.

\item The preferred longitude $\gamma $ is a simple closed
curve that generates the kernel of the homomorphism $H_{1}(X\cap
N)\rightarrow H_{1}(X)$.
\end{itemize}
\end{theorem}
From this theorem we also obtain the following corollary.

\begin{corollary}
The median $\eta$ is characterized by a simple closed curve on $X\cap N$ that bounds a disk in $N$. On the other hand, the preferred longitude $\gamma $ is characterized by a simple closed curve on $X\cap N$ that bounds a surface in $X$.
\end{corollary}

It is important to notice that once we choose the natural orientation on the meridian and longitude, as in the construction above, these curves are unique on $T^2$ up to ambient isotopy.

\begin{proposition}
\label{unique}
Let $K$ be an oriented knot in $S^3$. Let $N$ and $X$ be as before. There exist two oriented curves $\eta $ and $\gamma $ unique up to ambient isotopy on $T^2=X\cap N$ that satisfy Definition~\ref{Merid-long}.
\end{proposition}
\begin{proof}
The existence of the curves has already been established. For the uniqueness suppose that $\eta^{\prime}$ is another curve on $X\cap N$ with the same properties of $\eta$. Recall that the curve $\eta$ was a representative of a certain homology class in $H_1(T^2)$ and this homology class is a homology class that generates the kernel of the map $H_{1}(X\cap N)\rightarrow H_{1}(N)$ and this kernel is isomorphic to $\mathbb{Z}$. Hence each of the curves $\eta$ and $\eta^{\prime}$ must be a representative of a generator of kernel and hence $[\eta]=\pm[\eta^{\prime}]$. Now recall the construction of meridian above and notice that we can choose that orientation of $\eta$ and $\eta^{\prime}$ so that $lk(K,\eta)=lk(K,\eta^{\prime})=1$. Now with this choice we must have $[\eta]=[\eta^{\prime}]$. Since $\eta$ and $\eta^{\prime}$ are simple closed curves, again by construction, we conclude, by Theorem \ref{ambient}, that $\eta$ and $\eta^{\prime}$ are ambient isotopic. Similarly suppose that curve $\gamma^{\prime}$ is a curve on $X\cap N$ with the same properties as those of $\gamma$. These two curves are representatives of a generator of the kernel of the map $H_{1}(X\cap N)\rightarrow H_{1}(X)$ and hence $[\gamma]=\pm[\gamma^{\prime}]$. The orientation of two curves can be chosen so that they are both parallel to the oriented knot $K$ then we conclude that $[\gamma]=[\gamma^{\prime}]$ and thus, by Theorem \ref{ambient}, the curves $\gamma^{\prime}$ and $\gamma^{\prime}$ are ambient isotopic.
\end{proof}

\begin{remark}
It is worth noting that while a meridian can be defined for a solid torus, a preferred longitude requires a specified embedding of the solid torus into $S^{3}.$
\end{remark}

\begin{remark}
\label{long}
The preferred longitude $\gamma $ is \textit{not} determined
completely by stating that it is a simple closed curve on $\partial N$ that
generates $H_{1}(N).$ Actually there are infinitely many homology classes of curves on $\partial
N$ with this property. In fact a curve on $\partial N$ that generates $H_{1}(N)$ and is positively oriented with the knot is usually referred to by a \textit{longitude curve}. Note that there are infinitely curves on $\partial N$, up to homotopy, that satisfy this condition. On the other hand, adding the condition that this curve  is also trivial
in $H_{1}(X)$ determines that curve uniquely up to ambient isotopy on $\partial N,$ as we have shown in Proposition \ref{unique}. This also explains the adjective "preferred" when we want to describe the preferred longitude to distinguish this curve amongst many other longitude curves on $\partial N$. 
\end{remark}

\subsubsection{Different characterizations of the meridian and the preferred longitude}

It is useful to 
have many characterization for the meridian and the longitude. The following theorem summarizes most of the characterization of the meridian.

\begin{theorem}
Let $K$ be an oriented knot in (oriented) $S^{3}$ and let $X$ and $N$ be
defined as before. Suppose that $\eta $ is essential in $\partial N,$ then
the following are equivalent:

$(1)$ $\eta $ is homologically trivial in $N,$

$(2)$ $\eta $ is homotopically trivial in $N,$

$(3)$ $\eta $ bounds $\ $a disk in $N.$

\end{theorem}

The choice of a meridian of a knot does not include any ambiguity. However,
the choice of a preferred longitude needs more care. There is an easy
characterization for the preferred framing given in terms of the linking number. This characterization is given in the following theorem.




\begin{theorem}
\label{per long}
The preferred longitude $\gamma$ of a knot $K$
in $S^{3}$ is characterized by a simple closed curve on $N$ such that $
lk(\gamma ,K)=0$.
\end{theorem}

\begin{proof}
Viewing $[\gamma ]_{X}$ as an element in $X$ we can write $[\gamma ]_{X}=n$ $[\eta ]$ where
$[\eta ]$ is the generator of $H_{1}(X)$ and $n$ is some integer. The
integer $n$ is, by the definition of the linking number, $lk(\gamma ,K)$.
If $[\eta ]$ is a preferred longitude then by definition $[\gamma ]_{X}=0$ and hence $lk(\gamma ,K)=0$. On the other hand, if $lk(\gamma ,K)=0$ then $[\gamma ]_{X}=0$ and hence the result follows.
\end{proof}

\subsection{The relation between the longitudes and the framings of a knot }
\label{herethere}
In this section we 
relate the notions of longitudes and framings of a knot.
Let $K$ be a knot in $S^{3}.$  Every framing of $K$ gives rise to a longitude of $K$ on $X\cap N$ and vise versa. We first show that the zero-framing corresponds to the preferred longitude.

Choose an orientation of the knot $K$. We know that $H_{1}(S^{3}\backslash
K)=H_{1}(X)=\mathbb{Z}.$ We can pick the generator to be $[\eta ]$ the meridian of the tubular
neighborhood around $K$. Choose a framing $V$ for $K$. We know that this framing
gives rise to another knot $K^{\prime }$ that is linked with $K$
and the linking number between $K$ and $K^{\prime }$ is precisely the framing integer determined by the framing $V$. Now the curve $K^{\prime }$ represents an
element of $H_{1}(S^{3}\backslash K)=H_{1}(X)$ and hence it can be written
as $m[\eta ]_{X}$ for some integer $m$. We
conclude that every framing corresponds to some integer $m$ in the homology
of the exterior of the knot $K$. In particular the zero-framing
corresponds to the integer $0$ and hence the linking number zero. Thus, by Theorem \ref{per long}, the zero-framing of a knot $K$ corresponds to the preferred longitude $\gamma$ of a tubular neighborhood of the knot $K.$ We have proven the following theorem.

\begin{theorem}
\label{here}
Let $K$ be a zero-framed knot in $S^{3}$. Suppose that $N$ is a tubular
neighborhood of the knot $K\ $that intersects the ribbon of $K$ in a simple
closed curve $\gamma$.  Then $\gamma $ is the preferred longitude of a tubular neighborhood of the knot $K.$
\end{theorem}

This theorem can be generalized to characterize any framing for a given knot. To see this let $K$ be a framed knot and let $N$
be its tubular neighborhood and $X$ its exterior.  Then $K$ intersects the
torus $\partial N$ in a simple closed curve, say $d$ that winds $m^{\prime }$
times around the meridian and $1$ time around the preferred longitude.  Thus it can
be represented by
\begin{equation}
\label{ggg}
   d=m^{\prime }[\eta ]+[\gamma ]. 
\end{equation}
See also Figure ~\ref{framedtorus123}.
We want to show that $m^{\prime }$
 is precisely the framing integer of $K$. Recall that the framing number
is the self-linking number of $K$ which is by definition $lk(d,K).$

\begin{figure}[htb]
  \centering
   {\includegraphics[scale=0.12]{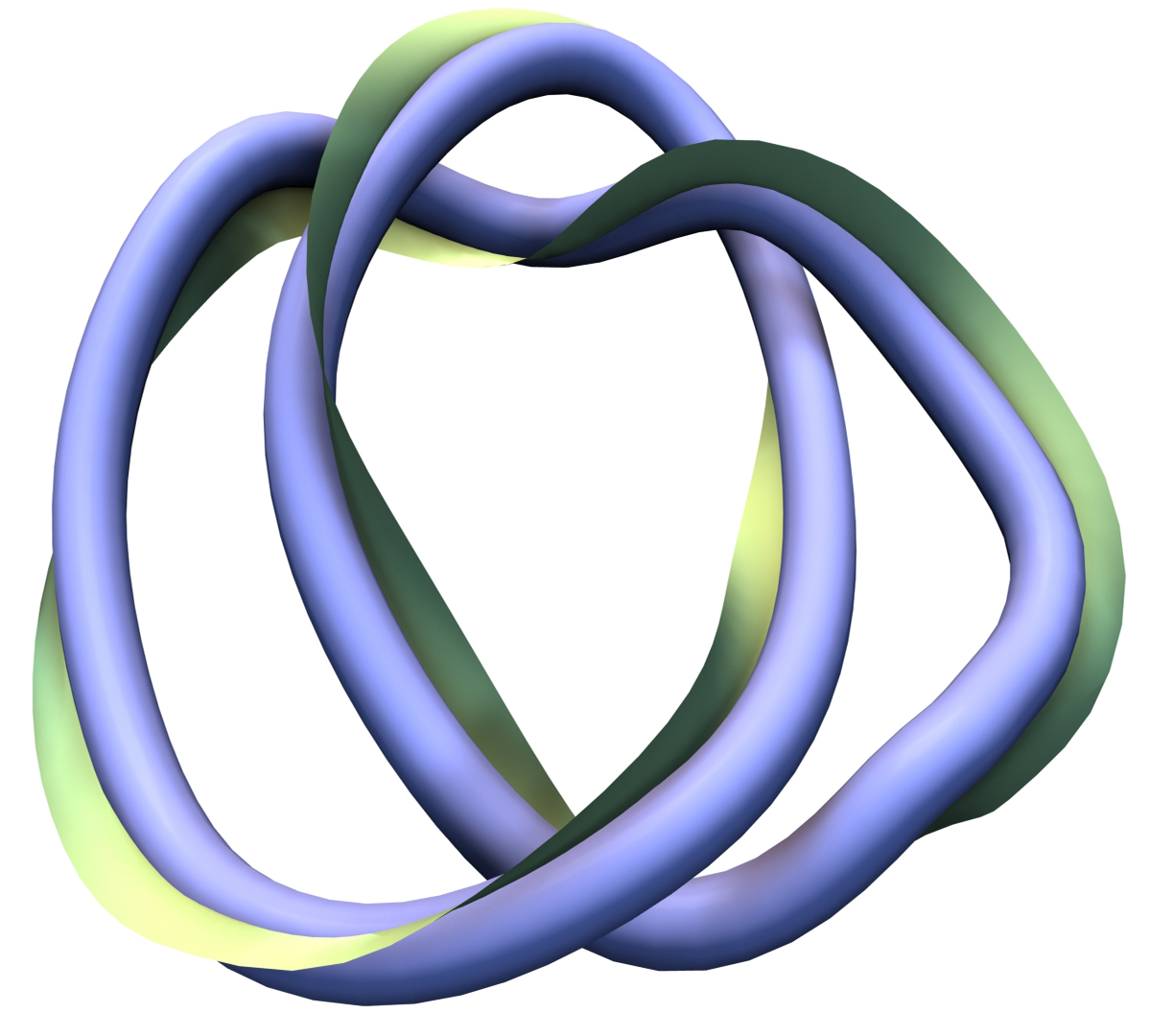}
  \caption{The curve $d$ is obtained by the intersection of the framed knot $K$ with the torus $\partial N$.}
  \label{framedtorus123}}
\end{figure}

To this end consider the image of the curve $d$ under the isomorphism $i_{\ast }\oplus j_{\ast }$. This can be seen to be  $(i_{\ast
}\oplus j_{\ast })(d)=(i_{\ast
}\oplus j_{\ast })(m^{\prime }[\eta ]+[\gamma ])=m^{\prime }(i_{\ast }\oplus
j_{\ast })([\eta ])+(i_{\ast }\oplus j_{\ast })([\gamma ])=m^{\prime }(i_{\ast
}([\eta ]),j_{\ast }([\eta ]))+(i_{\ast }([\gamma ]),j_{\ast
}([\gamma ]))=m^{\prime }(1,0)+(0,1)=(m^{\prime },1)=(m^{\prime }i_{\ast
}[\eta ],j_{\ast }([\gamma ])\in H_{1}(X)\oplus H_{1}(N)$ and thus $
[d]_{X}=m^{\prime }i_{\ast }[\eta ]=m^{\prime }[\eta ]_{X}$.  Hence, by
the definition of the linking number, $m^{\prime }$ must be $lk(d,K)$ and we
are done.

There is a little more that we can say about the curve $d$ given in \ref{ggg}. Recall from Remark \ref{long} that a longitude curve on $\partial N$ is a curve that generates $H_{1}(N)$ and is positively oriented with the knot $K$. We have proved that the curve  curve $d$ defined in \ref{ggg} actually satisfies this condition : it is obtained from the intersection of a framed knot with $\partial N$, and it generates $H_{1}(N)$ since it is $1$ in this group as we have shown. Hence this a longitude curve can be written in terms of the meridian and the preferred longitude as $m[\eta ]+[\gamma ]$ for some $m\in \mathbb{Z}$. In other words, the curve $d$ as given in \ref{ggg} parameterizes all longitude curves on the $\partial N$. We record this in the following Theorem.

\begin{theorem}
\label{final}
Let $(K,V)$ be a framed knot in $S^{3}$. Suppose that $N$ is a tubular
neighborhood of the knot $K\ $that intersects the ribbon of $K$ in a simple
closed curve $d$. Then $d$ is the a longitude curve on $\partial N$. Furthermore $d$ can be written as $d=m[\eta ]+[\gamma ]$ where $m\in \mathbb{Z}$ is the framing integer of $(K,V)$ and $\eta$ and $\gamma$ are the  meridian and the preferred longitude respectively. Conversely, any longitude curve on $\partial N$ can be written as $m[\eta ]+[\gamma ]$ for some $m \in \mathbb{Z}$ and it corresponds to a framing $V$ of the knot $K$ whose framing integer equals to $m$. 
\end{theorem}






\section{Seifert surfaces and zero-framed knots}

In this section we give the Seifert framing which is a type of framing that can be associated with a knot $K$. We prove that this framing can be used to characterize the zero framing of a knot.

\begin{definition}
Given a Seifert surface for a knot, the associated \textit{\textbf{Seifert framing}}
is obtained by taking a vector field perpendicular to the knot and inward tangent to the Seifert surface.
\end{definition}

The Seifert framing provides a useful characterization for the zero-framing of a knot.

\begin{theorem}\label{thm5.2}
The self-linking number obtained from a Seifert framing is always zero.
\end{theorem}

\begin{proof}
Suppose that $N$ is a tubular neighborhood of a knot $K$ and $X$ its
exterior. Let $S$ be the Seifert surface of $K$ and let $K^{\prime }$ be the
intersection curve $\partial N\cap S.$ It is clear that $K^{\prime }$ is a
simple closed curve on $\partial N.$ The curve $K^{\prime }$ bounds the
Seifert surface $S$ in $X$ and hence it is trivial in $H_{1}(X).$ Thus, $%
K^{\prime }$ is precisely the preferred longitude and by Theorem \ref{here} we conclude that $lk(K^{\prime },K)=0$. Hence the framing
obtained from the Seifert surface is zero.
\end{proof}

Alternatively, 
Theorem~\ref{thm5.2} can be seen to be true by utilizing a different definition of the linking number. Namely, let $K$ and $K^{\prime}$ as stated in Theorem~\ref{thm5.2}. 
Recall that $lk(K,K^{\prime })=S \cdot K^{\prime }$ where $S \cdot K^{\prime }$ is the intersection number between the surface $S$ and the knot $K^{\prime }$. From the way we construct $K^{\prime }$ we see the intersection number between $S$ and $K^{\prime }$ is zero. It is worth mentioning here that even though it looks as if there are infinitely many intersections between $S$ and $K^{\prime }$, these intersections are not transverse intersections and hence they do not contribute to the number $S \cdot K^{\prime }$. In other words, one needs to push the surface $S$ a little bit away back from the knot $K^{\prime }$ so that it does not intersect with $K^{\prime}$.

\section{Framing characterization using the fundamental group of $SO(2)$}

Another characterization of the framing of a knot is obtained by considering the fundamental group $\pi_{1}(SO(2))$ of the special orthogonal group $SO(2)$.  While less intuitive, this description will hint at how exactly we can make precise this notion of what it means to count ``twists'' around a knot.  
Recall that $SO(2)$ is the group of all $2\times2$ real matrices with determinant equal to $1$, and that geometrically these linear maps comprise the set of rotations in the plane about the origin.  As such, the group is topologically a circle, and can be parametrized by an angle $\theta$ corresponding to the angle of the rotation.  

Suppose that we are given a knot $K$ and a vector field $V$ representing the zero-framing of $K$, and choose an arbitrary orientation on the knot. We can consider this as a sort of reference framing that is used to create a well-defined map from the set of framings of $K$ into $\pi_{1}(SO(2))$.  For every point on $K$, construct a vector $N$ so that the ordered set $(\vec K,\vec V,\vec N)$ forms a right-hand basis (RHB) of $\mathbb{R}^{3}$. In this context, we treat $\vec K$ as a nonzero vector tangential to the knot, with the direction dictated by the knot's orientation. As such, it suffices to choose $\vec N = \vec K \times \vec V$ to get the desired RHB.  The knot $K$ is an embedded circle,  so we can identify points of the knot with their preimages in this embedding to get a parametrization $t \rightarrow e^{2\pi it} \rightarrow K(t)$ in terms of the standard unit circle $S^{1}$.  At every point $K(t)$, the associated vectors $\vec V(t)$ and $\vec N(t)$ lie in a plane perpendicular to the knot (see Figure ~\ref{basis}). 

Let $(K,W)$ be any choice of framing on K.  Using the same process as before, we associate to every point $K(t)$ a pair of nonzero orthogonal vectors $(\vec W(t), \vec M(t))$ that span the plane normal to the knot.  Choose an element $0 \leq \theta(0) < 2\pi$ of $SO(2)$ that represents the rotation of the basis $\vec V(0),\vec N(0)$ around the axis given by $\vec K(0)$ to obtain $(\vec W(0),\vec M(0))$.  Proceeding \textit{backwards} along the knot, we can create a smooth map $\phi_{W}:t \rightarrow \theta(t)$ that encodes the rotation of the framing $W$ with respect to $V$.  Note that $\phi(1) = 2\pi m$ must be a multiple of $2\pi$, because $\vec V(0) = \vec V(1)$ and $\vec W(0) =\vec W(1)$.  This multiple is exactly the framing integer.  

\begin{remark}
The decision to construct $\phi_{W}$ by proceeding backwards along the knot is made to ensure consistency with other characterizations of the framing integer given in this text, such as self-linking number.  We could alternatively have defined the map $\phi_{W}$ by proceeding forwards along the knot and negating the value $m$.  Or we could have chosen to reverse the parametrization of SO(2).  It is worth taking a minute to consider what other choices were made in this construction that could have the effect of changing the framing integer. 
\end{remark}

\begin{remark}
The reason for presenting this characterization of a framing as an element of the fundamental group of $SO(2)$ (ie, the circle) rather than simply stopping at the framing integer is because this interpretation allows us to easily see that a framing $W'$ constructed from $W$ by cutting the ribbon and giving it a full twist before reconnecting the ends of the ribbon is indeed distinct from $w$.  The effect of this operation would  add/subtract $2\pi$ to the value of $\phi_{W}(1)$, thus creating a distinct element of the fundamental group of SO(2). This is exactly analogous to the traditional calculation of the fundamental group of the circle using covering maps.  Any ambient isotopy that could map $W$ to $W'$ would induce a homotopy between $\phi_{W}$ and $\phi_{W'}$; since we know they are not homotopic, no such isotopy can exist. 
\end{remark}
On the other hand a loop $f$ in $SO(2)$ gives rise to a continuous family of elements in $SO(2)$ which can be used to construct a smooth vector field on $K$.  Perturbing the curve $f$ inside $S^1$ in a way that respects its homotopy type will change the vector field $V$ only up to some ambient isotopy. This yields a bijection between elements of $\pi_{1}(SO(2))$, and the set of equivalence classes of framings of $K$. 

\begin{figure}[htb]
  \centering
   {\includegraphics[scale=0.2]{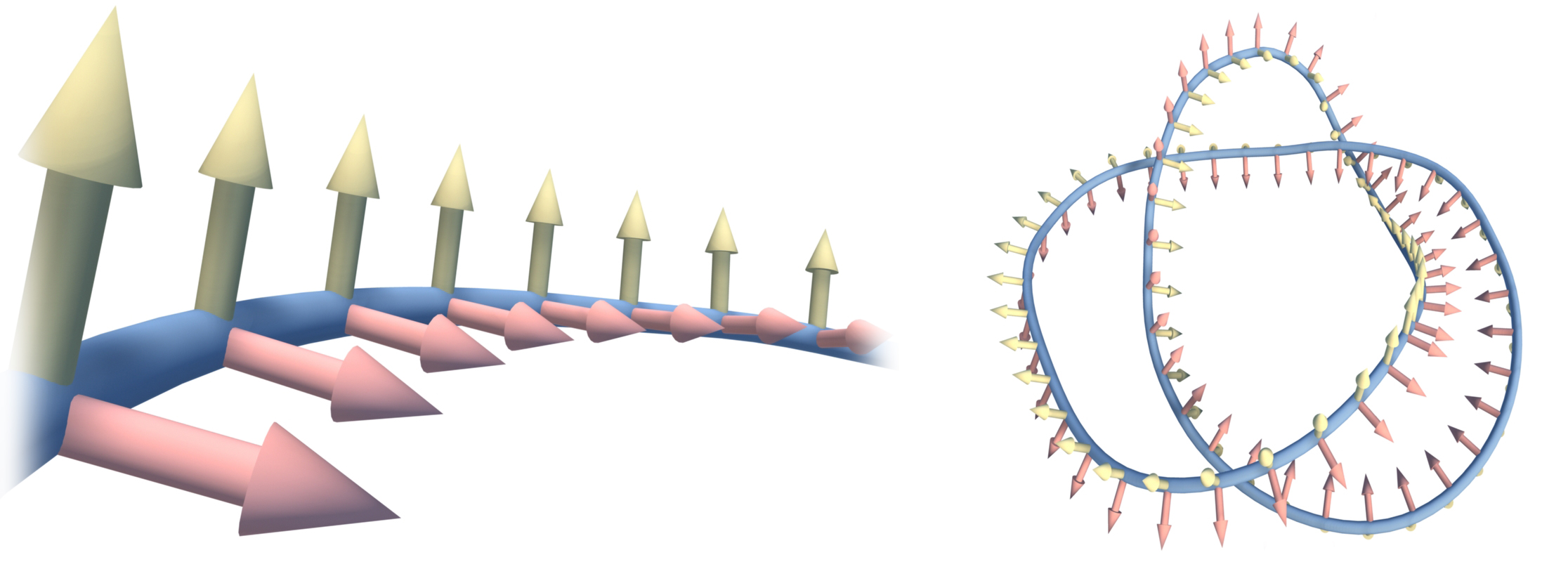}
	 \put(-100,-20){(B)}
	 \put(-350,-20){(A)}
  \caption{ (A) Choosing an orthogonal basis at every point on the knot.  The yellow vectors represent the framing of the knot and the red vectors are chosen to be orthogonal to both the framing and the knot itself. (B) An example of a choice of orthogonal bases on every point in knot.}
  \label{basis}}
\end{figure}

\section{Framed knots and $3$-manifolds}


In the introduction of this article we mentioned a theorem of Lickorish-Wallace (Theorem~\ref{LW} below). In this section we will give the necessary ingredients needed to state the theorem and illustrate how framed knots are utilized in lower-dimensional topology. 

Lickorish was trying to answer a question posed by Bing \cite{Bing} who gave a partial solution to the Poincar\'{e} conjecture.  Bing's question states: "Which compact, connected $3$-manifolds can be obtained from the $3$-sphere using the following process: deleting a disjoint polyhedral tori and sewing them back in a different way". As we will see in this section, all closed and orientable $3$-manifolds can be realized in this way and, importantly for us, framed knots play a central role in this.

\begin{theorem}\cite{Lickorish}\label{LW}
Any closed, orientable, connected $3$-manifold can be realized as integer surgery on some framed link in the $3$-dimensional sphere $S^{3}.$
\end{theorem}
Intuitively, Theorem \ref{LW} can be used in the fundamental quest of 3-manifold topology, to obtain a complete classification of all compact orientable 3-manifolds.
The problem is that this list might have redundancies : two different framed links may correspond to same $3$-manifold. To determine when two links give rise to the same $3$-manifold Kirby \cite{Kirby} studied the necessary moves on framed links, similar to Reidemeister moves, called Kirby moves. More precisely,  Kirby Calculus \cite{Kirby} states that two framed links produce the same $3$-manifold if and only if the links are related by a sequence of moves called Kirby moves.  Thus Kirby calculus in conjunction with Theorem \ref{LW} can be used in the fundamental quest of $3$-manifold topology, to obtain a complete classification of all compact orientable 3-manifolds.

 The term \textit{surgery} mentioned in Theorem \ref{LW} refers to the idea of performing "surgery" on a $3$-manifold. Intuitively, a surgery operation on a $3$-manifold $M$ usually involves removing a manifold with boundary $N$ from $M$ to obtain a $3$-manifold $M^{\prime}$ with a boundary $\partial M$ and then gluing $N$ back to $M^{\prime}$ via a homeomophism $f :\partial N \longrightarrow \partial M$. Choosing the way we glue $N$ to $M$ may provide different $3$-manifolds. There are multiple types of surgeries on $3$-manifolds such as integer surgery and rational surgery. In the remaining part of the paper we will to talk briefly about integer and rational surgeries on the $3$-sphere.
 
 In order to gain intuition we start with a few simple examples to show how one can obtain $3$-manifolds by gluing "simpler" $3-$manifolds together. 
 
 Our first example is the $3$-sphere. It is intuitively clear that one can obtain the $2$-sphere $S^2$ by gluing two $2$-disks $D^2=\{(x,y) \in \mathbb{R}^2| \; x^2+y^2\leq 1\}$ along their boundaries $S^1$. This intuition can actually be generalized to the $3$-sphere and the reader may convince herself that gluing two $3$-disks, $D^3=\{(x,y,z) \in \mathbb{R}^3|x^2+y^2+z^2\leq 1\}$, along their boundaries $S^2$ gives the $3$-sphere $S^3$. See Figure \ref{Gluing to get sphere}. A key fact here is there is only one way to glue $D^3$ to itself. Roughly speaking, there is essentially only one homeomorphism between $S^2$ and itself \footnote{This is a result of Smale's Theorem which states that any orientation-preserving diffeomorphism of the $2$-sphere is smoothly isotopic to the identity map.}. 
 
 \begin{figure}[htb] 
  \centering
   {\includegraphics[scale=0.25]{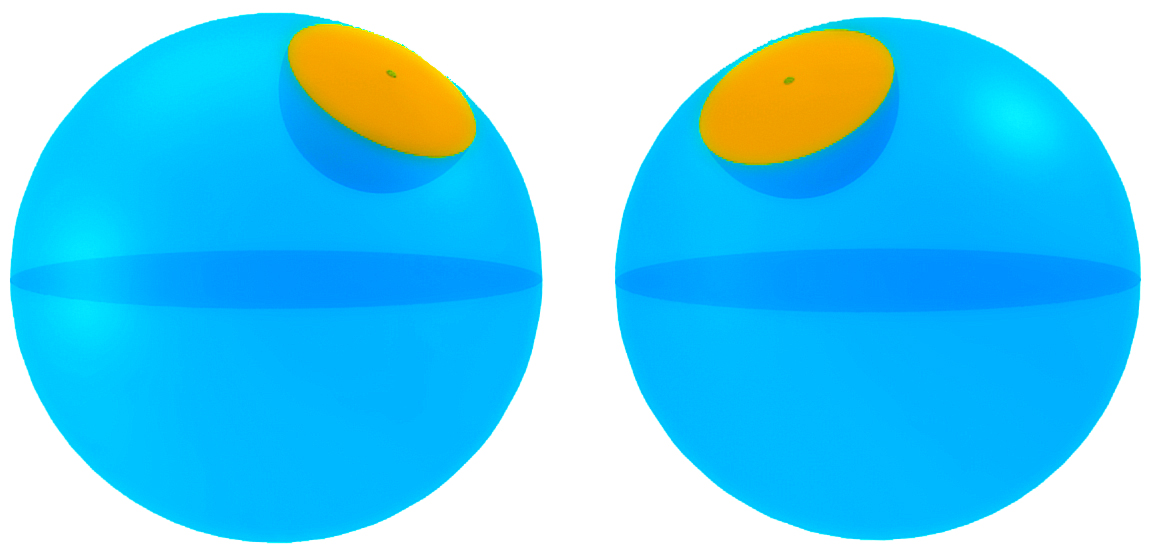}
   \put(-200,110){$p$}
      \put(-100,110){$p$}
  \caption{Gluing two disks $D^3$ along their boundaries $S^2$ gives the $3$-sphere $S^3$.  
A point $p$ on the first boundary is glued to a point $p$ on the second boundary.
Similarly, a small neighborhood of the point $p$ of can be visualized as being partially in the first ball and partially in the second as shown in the figure. In particular the  brown disk on the surfaces of the balls are identified and represent the same points in the $S^3$.}
  \label{Gluing to get sphere}}
\end{figure}

This exhausts the list of $3$-manifold that can be obtained from gluing $D^3$  to itself. We shall not prove this result here. The next natural choice of  simple $3$-manifolds that one can consider is the solid torus. What are the different manifolds that one can obtain by gluing two solid tori along their boundaries? It turns out that in this case we can obtain infinitely many manifolds! We describe this next.

Let $ST_1$ and $ST_2$ be two solid tori. Let $f : \partial ST_1 \longrightarrow \partial ST_2 $ be a homemophism, between their boundaries, that sends the meridian of $\partial ST_1$ to the longitude of $\partial ST_2$. It turns out that the manifold obtained by this gluing is again $S^3$. To see this, denote by $D^2$ the merdional disk of $ST_1$. We thicken the disk $D^2$ a little bit to obtain $D^2\times I$ and cut this part out of $ST_1$. We obtain in this way $D^2\times I$ and another piece that is homeomrphic to $D^3$. See Figure \ref{framedtorus} (A). Now, if we glue $D^2\times I$ to the solid torus $ST_2$ along the longitude, as indicated in Figure \ref{framedtorus} (B), one obtains back a space that is homeomrphic to $D^3$ as well. To finish the gluing process of $ST_1$ and $ST_2$, we need to glue the remaining boundaries together. However, the resulting two manifolds are exactly two $3$-disks and by our earlier discussion there is only one way to glue such two manifolds together along their boundaries. Hence the resulting manifold is again $S^3$. See Figure \ref{framedtorus} (C).

\begin{figure}[htb] 
  \centering
   {\includegraphics[scale=0.067]{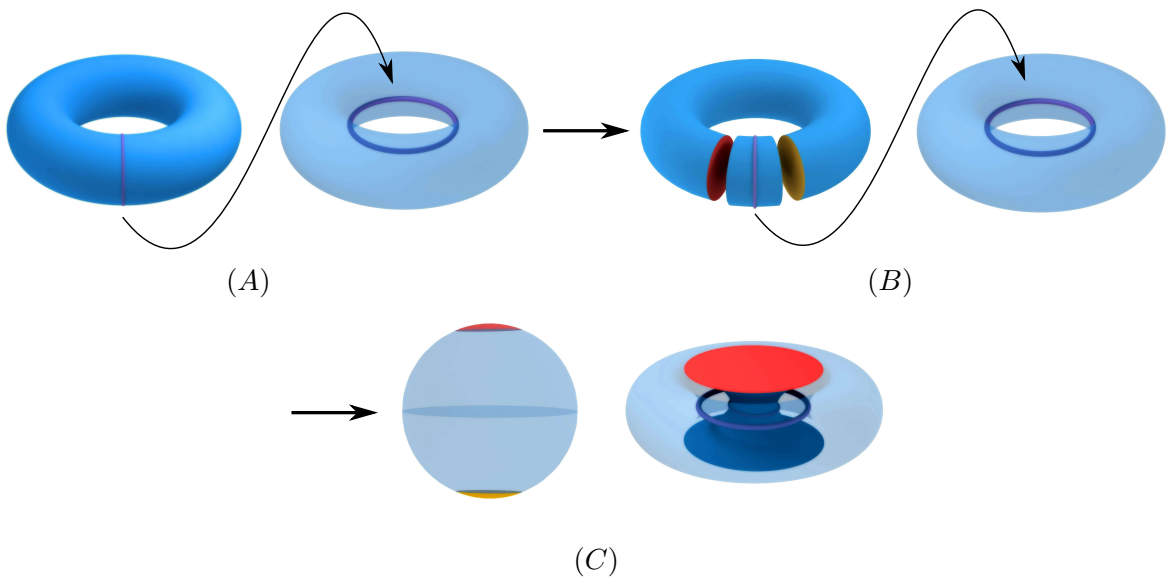}
      \put(-360,90){$(A)$}
 \put(-230,-15){$(C)$}
 \put(-120,90){$(B)$}
  \caption{Gluing two solid tori along the boundaries is determined where we send the meridian. $(A)$ The meridian of the first solid torus is sent to the longitude of the second solid torus. This process can be done in two steps which are shown in $(B)$ and $(C)$.  In $(B)$ we glue $D^2\times I$ to the longitude of the second solid torus. Finally, in $(C)$ we obtain two $3$-disks. Up to isotopy, any  orientation-preserving diffeomorphism of the $2$-sphere is smoothly isotopic to the identity map. Thus, the final manifold is completely determined and in this case it is the $3$-sphere..
  }
  \label{framedtorus}}
\end{figure}

The first that the reader should be aware of from the previous example is that that resulting manifold obtained from gluing $ST_1$ to $ST_2$ by sending the meridian of the first one to the longitude of the second one was completely determined by where we sent the meridian. It turns that that the resulting manifold is always completely determined by the image of the meridian under the gluing homeomorphism. But what are the the other 3-manifolds that one could obtain if we choose to map the meridian of $\partial ST_1$ to the another closed and simple (that is without self-intersection) curve on $\partial ST_2$?


  Recalling Theorem \ref{mainmain}, we can characterize simple closed  curves on the $ \partial ST_2$ by two coprime integers $p$ and $q$, where $p$ is the number of times the curve rounds around the meridian and $q$ are the number of times the curve rounds around the longitude. If we choose to map the meridian of $\partial ST_1$  to a curve $(p,q)$, where $p$ and $q$ are coprime, in $ \partial ST_2$ then the resulting 3-manifold is called a \textit{lens space} and it is denoted by $L(p,q)$. See Figure \ref{rational}.

Finally in the case when we map the meridian of $ST_1$ to the meridian of $ST_2$ then the resulting manifold is homeomorphic to $S^2\times S^1$. We shall not prove this fact here. The reader is referred to \cite{Adams,Rolfsen} for more details.

\begin{figure}[h] 
  \centering
   {\includegraphics[scale=0.08]{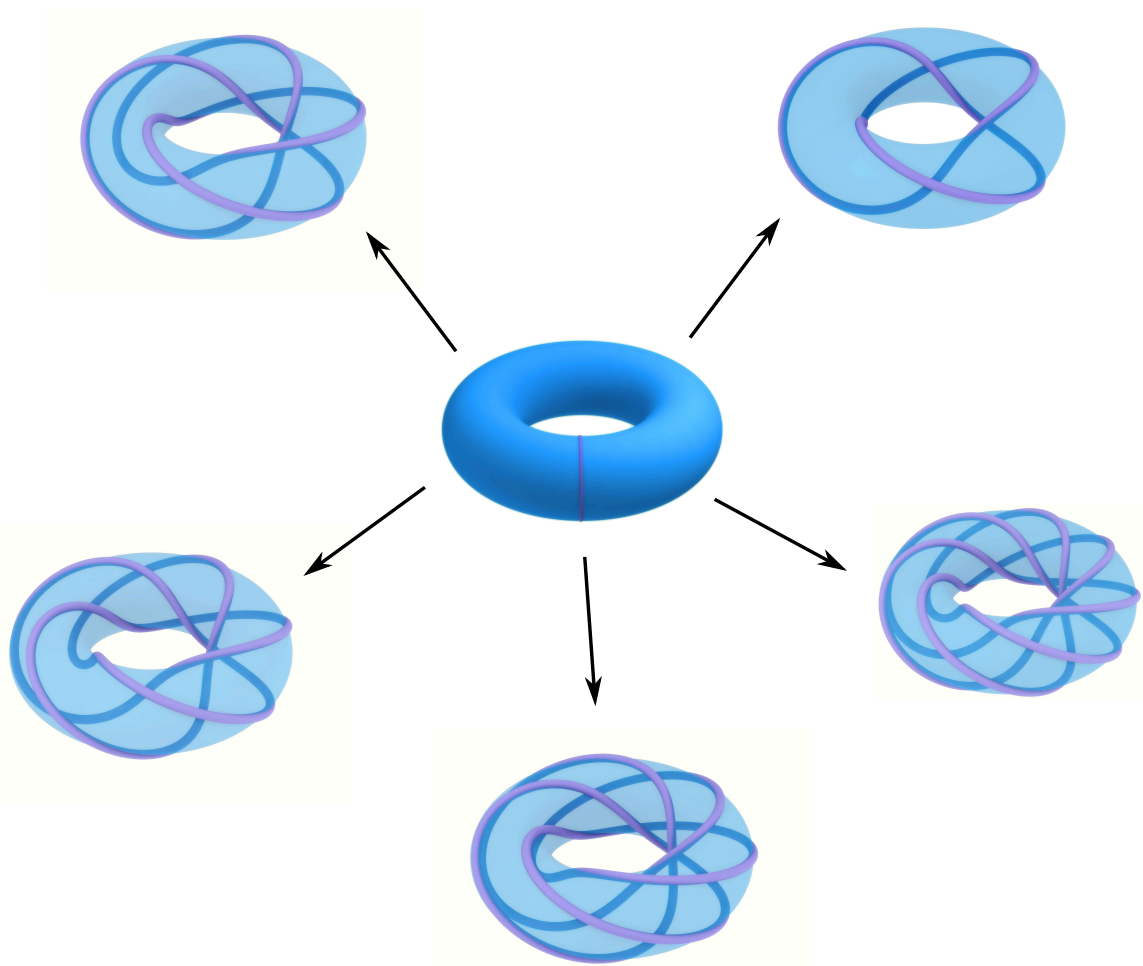}
  \caption{For $p, q$ coprime numbers, the Lens space $L(p,q)$ is obtained by gluing two solid tori along thier boundaries. The gluing map is determined by sending the meridian of of the boundary of the first solid torus to the simple closed curve $(p,q)$ on the boundary of the second solid torus. The figure shows various $(p,q)$ curves on the torus. For all such curves on the torus, and by sending the meridian of the first solid torus to these curves we can generate all lens spaces.}
  \label{rational}}
\end{figure}
 
 We are now ready to see how framed knots can be utilized in obtaining $3$-manifolds. Consider a knot $K$ in the $3$-sphere and let $N(K)$ be a tubular neighborhood of $K$ as before. By cutting open along the torus boundary $\partial N(K)$ of $N(K)$, we obtain the complement  $S^3 \setminus Int(N(K))$ which has a boundary that is homeomorphic to a torus.  Let $h$ be a homeomorphism between the boundaries of $D^2 \times S^1$ and $S^3 \setminus Int(N(K))$. Consider the $3$-manifold obtained by gluing $D^2 \times S^1$ to $S^3 \setminus Int(N(K))$ via the homeomorphism $h$. Just as before the final $3$-manifold  is completely determined by where we send the merdian of $D^2 \times S^1$.

 The $3$-manifold obtained is a closed orientable $3$-manifold and we say that it is obtained from the $3$-sphere by surgery along the knot $K$. This manifold, as we saw before, is completely determined by the image of the meridian. Up to isotopy we can assume that the meridian goes to a $(p,q)$-curve on the torus boundary of the knot where $p$ and $q$ are coprime. Moreover, it can be shown that the surgery that glues the meridian to the $(p,q)$-curve is the same as the surgery that sends the meridian to $(-p,-q)$-curve. Thus this surgery of the $3$-sphere is completely known by the fraction $\frac{p}{q}$ which is called the \textit{surgery index}. We call the above operation on $S^3$ a \textit{rational surgery} with rational index $r=\frac{p}{q}$. 
 
Our earlier discussion about lens spaces $L(p,q)$ imply that these spaces can be obtained by performing a rational surgery on the unknot. We now show how framed linked are naturally related to the notion of surgery.

 \subsection{Integer surgery and framed links}
 
 In this final part we briefly introduce integer surgery and then we show its relationship to framed links. First we state the definition of integer surgery.



\begin{definition}
If the integer $q$ is equal to $ \pm 1$, then we say that we have integer surgery on $S^3$.  
\end{definition}
We now explain the relationship between integer surgery and framed knots, recall that a framed knot $K$ determines a longitude curve on $\partial N$ where $N$ is the solid torus neighborhood of $K$. Moreover, as we illustrated earlier, this curve can be written as a $(p,1)$-curve in the torus where $p$ is the framing integer (Theorem \ref{final}). Hence the information given by a framed knot, the knot and its framing integer, is precisely the same information one needs to perform an integer surgery on $S^3$. Thus, given Theorem \ref{LW}, every compact orientable $3$-manifold can be represented by a link diagram with an integer on each link component. We have turned all of $3$-manifold theory into a version of knot theory!

 As an example of a $3$-manifold obtained from framed knots, recall from our earlier discussion that $S^2\times S^1$ can be obtained by by gluing two solid tori along their boundaries by sending the meridian of the boundary of the first solid torus one to the meridian of the boundary of the second solid torus. Another equivalent way to obtain the same $3$-manifold is to perform an integer surgery on $S^3$ along the zero-framed unknot. Notice here that when $K$ is the zero-framed unknot then the manifold  $S^3 \setminus Int(N(K))$ is homeomorphic to the solid torus. This solid torus is "flipped inside out". In particular, the longitude curve in  $S^3 \setminus Int(N(K))$ corresponds to a meridian curve in the standardly embedded solid torus.


Working with framed link diagrams in the context of three manifold is advantageous because one can utilize a link diagram with its blackboard framing as a well-defined method to denote the $3$-manifold obtained by performing the surgery on $S^3$ along that link. The blackboard framing of the link, along with the link diagram, completely determine the surgery and hence the manifold itself. Hence, framed link diagrams can be used to define $3$-manifold invariants. Indeed, any quantity defined on framed links diagrams that is invariant under $\Omega_2$ and $\Omega_3$ as well as Kirby moves can be considered as an invariant of $3$-manifolds. An interesting family of knots and $3$-manifold invariants called the quantum invariants has been at the center of interest in low-dimensional topology for decades now and framed knots play an important role in these invariants. The Jones polynomial \cite{jones1997polynomial} was the first invariant discovered from this family and then later Kauffman \cite{Kauffman} showed that this invariants can be defined via framed links. For more details about this subject see \cite{Ohtsuki}.

{\bf Acknowledgment:}  The authors would like to thank the referee for fruitful comments which improved the paper.


\end{document}